\documentclass[11pt]{amsart}

\usepackage{amsaddr}
\usepackage{amscd,amssymb}
\usepackage{hyperref}
\usepackage[all]{xy}

\title[Epsilon-Strongly Groupoid Graded Rings]{Epsilon-Strongly Groupoid 
Graded Rings, The Picard Inverse Category and Cohomology}

\newtheorem{thm}{Theorem}
\newtheorem*{thmnonum}{Theorem}
\newtheorem{prop}[thm]{Proposition}
\newtheorem{lem}[thm]{Lemma}

\theoremstyle{definition}
\newtheorem{defi}[thm]{Definition}
\newtheorem{exa}[thm]{Example}
\newtheorem{rem}[thm]{Remark}

\newcommand{\m}{{}^{-1}}
\newcommand{\N}{\mathbb{N}}
\newcommand{\image}{\rm{im\,}}
\DeclareMathOperator{\End}{End}
\DeclareMathOperator{\Aut}{Aut}
\DeclareMathOperator{\Hom}{Hom}
\DeclareMathOperator{\Pic}{Pic}
\DeclareMathOperator{\ob}{ob}

\begin{document}

\author{Patrik Nystedt}
\address{Department of Engineering Science,
University West, 
SE-46186 Trollh\"{a}ttan, Sweden}

\author{Johan \"{O}inert}
\address{Department of Mathematics and Natural Sciences,
Blekinge Institute of Technology,
SE-37179 Karlskrona, Sweden}

\author{H\'{e}ctor Pinedo}
\address{Escuela de Matem\'{a}ticas,
Universidad Industrial de Santander,
Carrera 27 Calle 9,
Edificio Camilo Torres
Apartado de correos 678,
Bucaramanga, Colombia}

\email{{\scriptsize patrik.nystedt@hv.se; johan.oinert@bth.se; hpinedot@uis.edu.co}}



\begin{abstract}
We introduce the class of partially invertible modules and show that it is an inverse category
which we call the Picard inverse category.
We use this ca\-te\-gory to generalize the classical construction of crossed
products to, what we call, generalized epsilon-crossed products and show that
these coincide with the class of epsilon-strongly groupoid graded rings.
We then use generalized epsilon-crossed groupoid products to obtain a generali\-zation,
from the group graded situation to the groupoid graded case, 
of the bijection from a certain second cohomology group, defined by the grading
and the functor from the groupoid in question to the Picard
inverse category, to the collection of equivalence classes of rings
epsilon-strongly graded by the groupoid.
\end{abstract}

\maketitle

\section{Introduction}\label{sec:intro}

Almost 40 years ago N{\v a}st{\v a}sescu and Van Oystaeyen \cite{nas1982}
proved an elegant result 
that relates the collection of group graded equivalence classes of strongly 
graded rings to certain second cohomology groups.
Namely, let $S$ be a ring.
We always assume that $S$ is associative and equipped
with a multiplicative identity $1_S$.
The ring $S$ is called \emph{graded} by a group $G$ (or \emph{$G$-graded}) 
if there is a set $\{ S_g \}_{g \in G}$ of additive subgroups of
$S$ such that $S = \oplus_{g \in G} S_g,$ and for all $g,h \in G$ the inclusion 
$S_g S_h \subseteq S_{gh}$ holds.
The ring $S$ is called \emph{strongly graded} if for all
$g,h \in G$ the equality $S_g S_h = S_{gh}$ holds.
Given two $G$-graded rings $S$ and $T$, a ring homomorphism 
$f : S \rightarrow T$ is called \emph{graded} if for all $g,h \in G$
the inclusion $f(S_g) \subseteq T_g$ holds.
Now, by \cite[Proposition I.3.13]{nas1982}
the collection of strongly graded rings can be parametrized by the so called 
\emph{generalized crossed products} $(A,F,f)$, for rings
$A$ and group homomorphisms $F$ from $G$ to the Picard group $\Pic(A)$ of $A$. 
Indeed, for each $g \in G$ we put $F(g) = [ P_g ]$ (the $A$-bimodule isomorphism
class of $P_g$) and we assume that $F(e) = [A],$ where $e$ denotes the identity element of $G$.
The map $f$ is a  factor set associated with $F$.
By this we mean a collection of $A$-bimodule
isomorphisms $f_{g,h} : P_g \otimes_A P_h
\rightarrow P_{gh}$ chosen so that the following diagram
commutes
\[
\CD P_g \otimes_A P_h \otimes_A P_p @> 
{\rm id}_{P_g} \otimes f_{h,p} >> P_g \otimes_A P_{hp} \\
@V f_{g,h} \otimes {\rm id}_{P_p} VV @VV f_{g,hp} V \\
P_{gh} \otimes_A P_{p}
@> f_{gh,p} >> P_{g h p} \\
\endCD
\]
for all $g,h,p \in G$.
The multiplication in $(A,F,f) = \oplus_{g \in G} P_g$ is defined by the biadditive extension of the
relation $x \cdot y = f_{g,h}(x \otimes y)$ for all 
$x \in P_g$, all $y \in P_h$ and all $g,h \in G$. 
Let $(A,F)$ denote the collection of all generalized crossed products
of the form $(A,F,f)$, where $f$ is a factor set associated with $F$.
Given $D$ and $D'$ in $(A,F)$ put 
$D \approx D'$ if there is an isomorphism of graded rings $D \rightarrow D'$
which is simultaneously an $A$-bimodule isomorphism.
Then $\approx$ is an equivalence relation on $(A,F)$ and
we can define $C(A,F) = (A,F) / \approx$.
Let $U(Z(A))$ denote the unit group of the center
$Z(A) = \{ a \in A \mid \forall b \in A, \ ab = ba \}$ of $A$.
The classical cohomology groups $H^n ( G , U(Z(A)) )$, for $n \geq 0$,
can then be defined, and in particular the corresponding second cohomology group.

\begin{thm}[N{\v a}st{\v a}sescu and Van Oystaeyen \cite{nas1982}]\label{maintheorem}
If $A$ is a ring, $F : G \rightarrow \Pic(A)$ is a group 
homomorphism and $f$ is a factor set associated with $F$,
then the map $H^2 ( G , U(Z(A)) ) \rightarrow C(A , F)$,
defined by $[q] \mapsto qf$, is bijective.
\end{thm}

Many natural examples of rings, such as rings of
matrices, crossed product algebras defined by separable extensions
and groupoid rings 
are not, in a natural way, graded by groups, but instead by groupoids
(see e.g. \cite{lundstrom2004,lundstrom2006} or
Section~\ref{Sec:ExEpsGraded}
of the present article).
Let $G$ be a groupoid, that is, a category where
each morphism is an isomorphism.
The
classes
of objects and morphisms
of $G$ are denoted by $G_0$ and $G_1$, respectively. 
If $g \in G_1$, then the domain and codomain of 
$g$ is denoted by $d(g)$ and $c(g),$ respectively.
We let $G_2$ denote the set of all pairs $(g,h) \in G_1 \times G_1$
that are composable, that is, such that $d(g)=c(h)$. 
From now on, assume that $G$ is small, that is,
such that $G_1$ is a set, and let $S$ be a ring which is \emph{graded} by $G$.
Recall from \cite{lundstrom2004,lundstrom2006} that this 
means that there is a set $\{ S_g \}_{g \in G_1}$ of additive subgroups of
$S$ such that $S = \oplus_{g \in G_1} S_g$ and, for all $g,h \in G_1$, 
$S_g S_h \subseteq S_{gh}$, if $(g,h) \in G_2$,
and $S_g S_h = \{ 0 \}$, otherwise.
In that case, $S$ is called \emph{strongly graded} if for all
$(g,h) \in G_2$ the equality $S_g S_h = S_{gh}$ holds.
By \cite[Proposition 4.1]{lundstrom2006}
the collection of strongly groupoid graded rings can be parametrized by 
generalized groupoid crossed products $(A,F,f)$, for rings
$A$ and groupoid homomorphisms $F$ from $G$ to the Picard groupoid PIC.
Here $A$ denotes the direct product of the rings $\{ F(e) \}_{e \in G_0}$.
Westman \cite{westman1971} (see also \cite{renault1980}) has developed a cohomology theory
for groupoids which extends the classical group cohomology theory.
In particular, the corresponding second cohomology group 
$H^2( G , Z(A) )$ can be defined.
The set $C(A,F)$ is defined analogously to the group graded case.

\begin{thm}[Lundstr\"{o}m \cite{lundstrom2006}]\label{maintheoremgroupoid}
If $G$ is a groupoid, $F : G \rightarrow {\rm PIC}$ is a functor 
and $f$ is a factor set associated with $F$,
then the map $H^2 ( G , U(Z(A)) ) \rightarrow C(A , F)$,
defined by $[q] \mapsto qf$, is bijective.
\end{thm}

In \cite{epsilon2018} the authors of the present article
introduced the class of epsilon-strongly group-graded rings.
This class properly contains both the class of strongly graded rings 
and the class of unital partial crossed products.
The main goal of the present article is to show a simultaneous 
generali\-zation (see Theorem~\ref{maintheoremepsilon}) of Theorem~\ref{maintheorem} 
and Theorem~\ref{maintheoremgroupoid} that holds for an 
even wider family of rings, namely the class
of epsilon-strongly groupoid graded rings.
Indeed, let $S$ be a ring which is graded by a small groupoid $G$.
We say that $S$ is \emph{epsilon-strongly graded by $G$}
if for each $g \in G_1$, $S_g S_{g^{-1}}$ is a unital ideal of $S_{c(g)}$
such that for all $(g,h) \in G_2$
the equalities $S_g S_h = S_g S_{g^{-1}} S_{gh} = S_{gh} S_{h^{-1}} S_h$ hold. 
Let $\epsilon_g$ denote the multiplicative identity element of $S_g S_{g^{-1}}$
and put $R = \oplus_{e \in G_0} S_e$.
In this context it turns out that the multiplication map 
$S_g \otimes_R S_h \rightarrow S_{gh}$, for $(g,h) \in G_2$,
is injective with image equal to $\epsilon_g S_{gh}$. 
In particular this implies that the $R$-isomorphism classes of the modules
$[S_g]$ do not form a groupoid. Instead they form an inverse category PIC$_{cat}$
which we call \emph{the Picard inverse category}.
The collection of epsilon-strongly groupoid graded rings can then be parametrized by 
genera\-lized groupoid epsilon-crossed products $(A,F,f)$, for rings
$A$ and partial functors of inverse categories $F : G \rightarrow {\rm PIC}_{cat}$.
Here $A$ denotes the direct product of the rings $\{ F(e) \}_{e \in G_0}$
and $f$ is a partial factor set.

On the other hand, in \cite{E-1} the concept of a partial action was introduced as an efficient tool to study $C^* $-algebras, permitting to characterize various important classes of them as crossed products by partial actions. The  study of partial actions  and partial representations of groups on algebras was initiated in \cite{DE}, and extended to the groupoid situation in \cite{bagio2012}. Recently, 
Dokuchaev and Khrypchenko \cite{DKh} have developed a cohomology theory
for partial actions of groups on monoids which we
extend to the groupoid setting.
In particular, we define the corresponding second cohomology group $H^2( G , Z(A) )$.
The set $C(A,F)$ is defined analogously to the group graded case, and we obtain the following theorem, which is the main result of this article.

\begin{thm}\label{maintheoremepsilon}
If $G$ is a groupoid,
$F : G \rightarrow {\rm PIC}_{cat}$ is a partial functor of inverse categories
and $f$ is a partial factor set associated with $F$,
then the map $H^2 ( G , U(Z(A)) ) \rightarrow C(A,F)$,
defined by $[q] \mapsto qf$, is bijective.
\end{thm}

Here is a detailed outline of this article.
In Section~\ref{preliminariescategories},
we state our conventions on categories. 
In particular, we recall the definition of inverse categories.
In that section, we also show that the collection of partial bijections
between sets form an inverse category which 
we denote by BIJ$_{cat}$ (see Proposition~\ref{BIJCAT}).
Inside this category sits the well-known groupoid BIJ of bijections between sets.
In Section~\ref{preliminariesrings}, 
we show (see Proposition~\ref{ISOISOC}) that the collection
of partial (commutative) ring isomorphisms
{\rm ISO}$_{cat}$ ({\rm ISOC}$_{cat}$) is an inverse subcategory of {\rm BIJ}$_{cat}$.
This category contains, in turn, the well-known
groupoid ISO (ISOC) having all (commutative) rings
as objects and ring isomorphisms as morphisms. We also show a result concerning
central idempotents in rings that we need in subsequent sections.
In Section~\ref{sectionpicardinversecategory},
we recall the definition of (pre-)equivalence data
and some properties of such systems, we also introduce the Picard inverse category PIC$_{cat}$
(see Definition~\ref{definitionpicardinversecategory} 
and Theorem~\ref{invcat}).
In Section~\ref{sectionepsilonstrongly},
we define epsilon-strongly groupoid graded rings
(see Definition~\ref{defepsilonstronglygradedring}).
In Section~\ref{Sec:ExEpsGraded}, we provide
examples of epsilon-strong groupoid gradings on partial skew groupoid rings,
Leavitt path algebras and Morita rings.
In Section~\ref{sectionepsilonmodules},
we define epsilon-strongly groupoid  graded modules
(see Definition~\ref{definitionepsilonmodules}) and we 
provide a characterization of them 
(see Proposition~\ref{tensorg}) that generalizes
a result previously obtained for strongly group graded modules.
At the end of the section,
after putting $R = \oplus_{e \in G_0} S_e$,
we show (see Proposition~\ref{tensor})
that the multiplication maps
$m_{g,h} : S_g \otimes_R S_h \rightarrow \epsilon_g S_{gh} = S_{gh} \epsilon_{h^{-1}}$,
for $(g,h) \in G_2$, are $R$-bimodule isomorphisms.
In particular this implies that for every $g \in G_1$, the sextuple
\begin{displaymath}
	( \epsilon_g R , \epsilon_{g^{-1}}R , S_g , S_{g^{-1}} , m_{g,g^{-1}} , m_{g^{-1},g} )
\end{displaymath}
is a set of equivalence data.
In Section~\ref{sectionepsiloncrossedproducts},
we introduce generalized epsilon-crossed groupoid products
(see Definition~\ref{defpartialgeneralizedepsiloncrossedproduct})
and we show that they parametrize the family of 
epsilon-strongly groupoid graded rings
(see Proposition~\ref{firstcorrespondence} and Proposition~\ref{secondcorrespondence}).
In Section~\ref{sectionpartialcohomology},
we extend the construction of a partial 
cohomology theory for partial actions of groups on 
commutative monoids, from \cite{DKh}, to partial actions of groupoids. 
In Section~\ref{proofmain},
we use the results in the previous sections 
to prove Theorem~\ref{maintheoremepsilon}.

\section{Preliminaries on categories}\label{preliminariescategories}

In this section, we state our conventions on categories. 
In particular, we recall the definition of inverse categories.
We introduce the new notion of a partial functor
of inverse categories (see Definition~\ref{defpartialfunctor})
and we show that the composition of two partial functors
of inverse categories is again a partial functor of inverse categories
(see Proposition~\ref{composition}).
In this section, we also show that the collection of partial bijections
between sets form an inverse category which 
we denote by BIJ$_{cat}$ (see Proposition~\ref{BIJCAT}).
Suppose that $G$ is a category.
The classes of objects and morphisms
of $G$ are denoted by $G_0$ and $G_1$, respectively. 
Recall that $G$ is called \emph{small} if $G_1$ is a set.
If $g \in G_1$, then the domain and codomain of 
$g$ are denoted by $d(g)$ and $c(g),$ respectively.
If $n \geq 2$, then we let $G_n$ denote the set of
all $(g_1,\ldots,g_n) \in \times_{i=1}^n G_1$
that are composable, that is, such that
for every $i \in \{ 1,\ldots,n-1 \}$ the relation $d(g_i) = c(g_{i+1})$ holds.
The category $G$ is called a \emph{groupoid} if  to each 
$g \in G_1$ there is a unique $h \in G_1$
such that $(g,h) \in G_2$, $(h,g) \in G_2$,
$g h = c(g)$ and $h g = d(g)$. In that case, we write $h = g^{-1}$.
A subcategory $H$ of a groupoid $G$ is called a \emph{subgroupoid}
if the restriction of the map $(\cdot)^{-1}$ on $G_1$
to $H_1$ coincides with the map $(\cdot)^{-1}$ on $H_1$.

Let $G$ be an inverse category.
Recall that this means that
there to each $g \in G_1$
is a unique $h \in G_1$ such that
$(g,h) \in G_2$, $(h,g) \in G_2$,
$g h g = g$ and $h g h = h$.
In that case we write $h = g^*$.
A subcategory $H$ of $G$ is called an \emph{inverse subcategory}
if the restriction of the map $*$ on $G_1$
to $H_1$ coincides with the map $*$ on $H_1$.
Note that if $G$ is a groupoid, then $G$ is an inverse
category if we for each $g \in G_1$ put $g^* = g^{-1}$.
It turns out that,  for our purposes, the usual notion of functor
is too restrictive when considered for inverse categories.
Therefore we make the following weakening.

\begin{defi}\label{defpartialfunctor}
Suppose that $G$ and $H$ are inverse categories.
A \emph{partial functor of inverse categories
$F : G \rightarrow H$} is a pair of maps $(F_0,F_1)$,
where $F_0 : G_0 \rightarrow H_0$ and $F_1 : G_1 \rightarrow H_1$,
that satisfy the following axioms.
\begin{itemize}

\item[(I1)] If $g : a \rightarrow b$ in $G_1$, then $F_1(g) : F_0(a) \rightarrow F_0(b)$.

\item[(I2)] If $a \in G_0$, then $F_1( {\rm id}_a ) = {\rm id}_{F_0(a)}$.

\item[(I3)] If $(g,h) \in G_2$, then 
\begin{displaymath}
	F_1(g) F_1(h) = F_1(g) F_1(g^*) F_1(gh) = F_1(gh) F_1(h^*) F_1(h).
\end{displaymath}

\end{itemize}
By abuse of notation, we will write $F$ for both $F_0$ and $F_1$ in the sequel.
\end{defi}

\begin{rem}
Note that if we replace axiom (I3) in Definition~\ref{defpartialfunctor} by 
\begin{itemize}

\item If $(g,h) \in G_2$, then $F_1(g) F_1(h) = F_1(gh)$.

\end{itemize}
then $F$ is an ordinary functor.
\end{rem}

\begin{prop}\label{ordinaryfunctorpartial}
If $G$ and $H$ are inverse categories and 
$F : G \rightarrow H$ is an ordinary functor,
then $F$ is a partial functor of inverse categories.
\end{prop}

\begin{proof}
Take $(g,h) \in G_2$. Then, since $F$ is an ordinary functor, we get that
$F(g) F(h) = 
F(gh) =
F(g g^* g h) =
F(g) F(g^*) F(gh)$ 
and
$F(g) F(h) = 
F(gh) =
F(g h h^* h) =
F(gh) F(h^*) F(h)$.
\end{proof}

\begin{prop}\label{propstar}
If $F : G \rightarrow H$ is a partial functor of 
inverse categories, then for every $g \in G_1$
the relation $F(g^*) = F(g)^*$ holds.
\end{prop}

\begin{proof}
Take a morphism $g : a \rightarrow b$ in $G_1$.
If we put $h = {\rm id}_a$ in (I3), then we get that
$F(g)F( {\rm id}_a ) = F(g) F(g^*) F(g)$.
From (I2) it follows that this relation can be written as
\begin{equation}\label{gggg}
F(g) = F(g) F(g^*) F(g).
\end{equation}
By replacing $g$ by $g^*$ in \eqref{gggg}, we get that
\begin{equation}\label{ggggg}
F(g^*) = F(g^*) F(g) F(g^*).
\end{equation}
Equations \eqref{gggg} and \eqref{ggggg} show that 
$F(g^*)$ satisfy the axioms for $F(g)^*$.
Since $F(g)^*$ is uniquely defined we thus get that $F(g^*) = F(g)^*$.
\end{proof}

\begin{lem}\label{lemmagh}
Suppose that $F : G \rightarrow H$ is a partial functor of inverse categories.
If $(g,h) \in G_2$ are chosen so that $gh = h$  $\left(gh = g\right),$
then $F(g)F(h) = F(h)$ $\left(F(g)F(h) = F(g)\right).$
\end{lem}

\begin{proof}
Suppose that $(g,h) \in G_2$ satisfy $gh = h$. From (I3)
and Proposition~\ref{propstar}, we get that
$F(g) F(h) = F(gh) F(h^*) F(h) = F(h) F(h)^* F(h) = F(h).$
Now suppose that $(g,h) \in G_2$ satisfy $gh = g$. From (I3)
and Proposition~\ref{propstar}, we get that
$F(g) F(h) = F(g) F(g^*) F(gh) = F(g) F(g)^* F(g) = F(g).$
\end{proof}

\begin{prop}\label{composition}
If $F : G \rightarrow G'$ and $F' : G' \rightarrow G''$
are partial functors of inverse categories, then
$F' \circ F : G \rightarrow G''$ is a partial functor of inverse categories.
\end{prop}

\begin{proof}
Put $F'' = F' \circ F$.
It is easy to see that (I1) and (I2) hold for $F''$.
Now we show (I3). To this end, take $(g,h) \in G_2$. 
We wish to show that
\begin{equation}\label{I4}
F''(g) F''(h) = F''(g) F''(g^*) F''(gh)
\end{equation}
and
\begin{equation}\label{I5}
F''(g) F''(h) = F''(gh) F(h^*) F(h).
\end{equation}
First we show \eqref{I4}.
Using (I3) for $F'$ and $F$, we get that the left side of \eqref{I4} equals
\begin{align*}
F'( F(g) ) F'( F(h) ) &= 
F'( F(g) ) F'( F(g)^* ) F'( F(g) F(h) )\\
&=  F'( F(g) ) F'( F(g)^* ) F'( F(g) F(g)^* F(gh) )
\end{align*}
Using Lemma~\ref{lemmagh} for  the right side of \eqref{I4} we get 
\begin{equation*}\label{gggh}
F'( F(g) ) F'( F(g)^* ) F'( F(gh) )=F'( F(g) ) [F'( F(g)^* ) F'( F(g)F(g)^* )] F'( F(gh) ).
\end{equation*}
Using (I3) and Lemma~\ref{lemmagh}, the last expression equals
\begin{align*}
F'( F(g) ) F'( F(g)^* ) [F'( F(g)F(g)^* ) F'( F(g)F(g)^* ) F'(  F(gh) )] \\
=  F'( F(g) ) F'( F(g)^* ) F'( F(g) F(g)^* F(gh) )
\end{align*}
which shows \eqref{I4}. Now we show \eqref{I5}.\\
Using (I3) for $F'$ and $F$, we get that the left side of \eqref{I5} equals
\begin{align*}F'( F(g) ) F'( F(h) )
& = 
F'( F(g)F(h) ) F'( F(h)^* ) F'( F(h) ) 
\\&= F'( F(gh) F(h^*) F(h) ) F'( F(h)^* ) F'( F(h) ).\end{align*}
The right side of \eqref{I5} equals
\begin{equation}\label{ghhh}
F'( F(gh) ) F'( F(h^*) ) F'( F(h) ).
\end{equation}
Using Lemma~\ref{lemmagh}, we get that \eqref{ghhh} equals
\begin{displaymath}
	F'( F(gh) ) F'( F(h^*) F(h) ) F'( F(h^*) ) F'( F(h) ).
\end{displaymath}
Using (I3) and Lemma~\ref{lemmagh}, the last expression equals
\begin{displaymath}
	F'( F(gh) F(h^*) F(h) ) F'( F(h^*) F(h) ) F'( F(h^*) F(h) ) F'( F(h^*) ) F'( F(h) )
\end{displaymath}
which equals 
$F'( F(gh) F(h^*) F(h) ) F'( F(h^*) ) F'( F(h) )$
showing \eqref{I5}.
\end{proof}

\begin{prop}
If $F : G \rightarrow H$ is a partial functor of inverse categories,
where $H$ is a groupoid, then $F$ is an ordinary functor.
\end{prop}

\begin{proof}
Take $(g,h) \in G_2$. From Proposition~\ref{propstar}, it follows that
$F(g) F(h) = F(g) F(g^*) F(gh) = F(g) F(g)^* F(gh) = F(g) F(g)^{-1} F(gh) = F(gh)$.
\end{proof}

\begin{defi}
Let BIJ denote the groupoid having 
the collection of all sets as objects and
bijections between sets as morphisms.
\end{defi}

\begin{defi}
Let $A$ and $B$ be sets.
By a \emph{partial bijection from $B$ to $A$} we mean
a choice of subsets $Y \subseteq B$ and $X \subseteq A$
and a bijection $f : Y \rightarrow X$.
We will indicate this by writing ${}_A^X f_B^Y$.
We will now define, what we call, \emph{the category of partial bijections} BIJ$_{cat}$.
The class of objects in BIJ$_{cat}$ consists of all sets. 
The class of morphisms in BIJ$_{cat}$ consists of all
partial bijections ${}_A^X f_B^Y$. 
The domain and codomain of such a morphism is $B$ and $A$, respectively.
The identity morphism at $A$ is defined to be
${}_A^A ( {\rm id}_A )_A^A$.
The empty partial bijection from $B$ to $A$ is denoted by $_{A} 0_B$.
Suppose that ${}_A^X f _B^Y $ and ${}_C^{X'} g _D^{Y'}$
are morphisms in BIJ$_{cat}$.
If $B=C$, then the composition of these morphisms is defined to be
\begin{displaymath}
	{}_A^{ f(Y \cap X') } f|_{Y \cap X'} \circ {g| _{ g^{-1}(Y \cap X')} }_D^ {g^{-1}(Y \cap X')} .
\end{displaymath}
Otherwise, the composition is defined to be $_{A}0_D$.
We also define a map $* : ({\rm BIJ}_{cat})_1 \rightarrow ({\rm BIJ}_{cat})_1$ by 
$( {}_A^X {f}_B^Y )^* = {}_B^Y {f^{-1}}_A^X $.
\end{defi}

\begin{prop}\label{BIJCAT}
{\rm BIJ}$_{cat}$ is an inverse category.
\end{prop}

\begin{proof}
Suppose that $\alpha =  {}_A^{X} f_B^Y $, 
$\beta = {}_B^{X'} g_C^{Y'} $ and $\gamma =  {}_C^{X''} h_D^{Y''} $
are morphisms in BIJ$_{cat}$. First we check the axioms for identity elements:
\begin{displaymath}
	{}_A^A ( {\rm id}_A )_A^A  \alpha = 
 {}_A^{ {\rm id}_A(A \cap X) } ( {\rm id}_A |_{A \cap X} \circ f|_{f^{-1}(A \cap X)} )_B^{f^{-1}(A \cap X)}  = 
 {}_A^X ( {\rm id}_X \circ f )_B^Y  = \alpha
\end{displaymath}
and
\begin{displaymath}
	\alpha  {}_B^B ( {\rm id}_B )_B^B  = 
 {}_A^{f(Y \cap B)} ( f|_{Y \cap B} \circ {\rm id}_B|_{ {\rm id}_B^{-1} 
(Y \cap B) } )_B^{ {\rm id}_B^{-1} (Y \cap B) }  =
 {}_B^Y ( f \circ {\rm id}_Y )_A^X  = \alpha.
\end{displaymath}
Now we prove associativity. We get that
\begin{displaymath}
	( \alpha \beta ) \gamma = 
 {}_A^{X_1}
( f|_{Y \cap X'} \circ g|_{g^{-1}( Y \cap X' )} )|_{ g^{-1}(Y \cap X') \cap X'' } 
\circ h|_{{Y_1}_{C}^{Y_1} },
\end{displaymath}
where 
\begin{displaymath}
	X_1 = ( f|_{Y \cap X'} \circ g|_{g^{-1}( Y \cap X' )} ) ( g^{-1}(Y \cap X') \cap X''),
\end{displaymath}
and
\begin{displaymath}
	Y_1 = h^{-1}( g^{-1}(Y \cap X') \cap X'' ).
\end{displaymath}
We also get that
\begin{displaymath}
	\alpha ( \beta \gamma) =
 {}_A^{ X_2 }  f|_{ Y \cap g(Y' \cap X'') } \circ
{ ( g|_{Y' \cap X''} \circ h|_{ h^{-1}(Y' \cap X'') } )|_{Y_2}}_D^{Y_2}
\end{displaymath}
where 
\begin{displaymath}
	X_2 = f( Y \cap g(Y' \cap X'') )
\end{displaymath}
and
\begin{displaymath}
	Y_2 = ( g|_{Y' \cap X''} \circ h|_{ h^{-1}(Y' \cap X'') } )^{-1} (Y \cap g(Y' \cap X'')).
\end{displaymath}
Since composition of functions is associative, 
we only need to show that $X_1 = X_2$ and $Y_1 = Y_2$.

First we show that $X_1 \subseteq X_2$. 
Take $a \in X_1$. 
Then $a = f( g ( b ) )$ for some $b \in X''$ such that $g(b) \in Y \cap X'$.
Since $g : Y' \rightarrow X'$ we also get that $b \in Y'$.
Thus $g(b) \in Y \cap g(Y' \cap X'')$ and hence $a = f ( g(b) ) \in X_2$.
Next we show that $X_2 \subseteq X_1$.
Take $c \in X_2$.
Then $c = f ( g ( d ) ),$ for some $d \in Y' \cap X''$ such that $g(d) \in Y$.
Since $g : Y' \rightarrow X'$ we get that $g(d) \in Y \cap X'$.
Thus $d \in g^{-1}( Y \cap X') \cap X''$ and hence $c \in X_1$.

Now we show that $Y_1 \subseteq Y_2$.
Take $a \in Y_1$.
Then $h(a) \in g^{-1}(Y \cap X') \cap X'' \subseteq Y' \cap X''$.
Thus $g(h(a)) \in g( Y' \cap X'' )$.
Also $g(h(a)) \in g ( g^{-1} (Y \cap X' ) ) \subseteq Y$.
Hence $g(h(a)) \in Y \cap g( Y' \cap X'' )$.
So we get that $a \in Y_2$.
Next we show that $Y_2 \subseteq Y_1$:
\begin{align*}Y_2 &= ( h^{-1} |_{h^{-1}(Y' \cap X'')}\circ g|_{Y' \cap X''}^{-1} ) (Y \cap g(Y' \cap X''))\\&
= h^{-1} |_{h^{-1}(Y' \cap X'')}
(  g^{-1}(Y \cap g(Y' \cap X'')) \cap Y' \cap X'' )
\\&= h^{-1}( g^{-1}(Y \cap g(Y' \cap X'')) \cap Y' \cap X'' ) \cap h^{-1}(Y' \cap X'').\end{align*}
Since $g(Y' \cap X'') \subseteq X'$ and $Y' \cap X'' \subseteq X''$ we thus get that
$Y_2 \subseteq Y_1$.

Finally, we show that BIJ$_{cat}$ is an inverse category.
To this end, first note that
\begin{equation}\label{firstset}
\alpha \alpha^* = 
{}_A^X f_B^Y  ( {}_A^X f_B^Y )^* =
{ }_A^X f_B^Y   {}_B^Y {f^{-1}}_A^X  =
 {}_A^X {{\rm id}_X}_A^X 
\end{equation}
and
\begin{equation}\label{secondset}
\alpha^* \alpha = 
( {}_A^X f_B^Y )^* {}_A^X f_B^Y  =
{}_B^Y {f^{-1}}_A^X  {}_A^X f_B^Y  =
 {}_B^Y {{\rm id}_Y}_B^Y .
\end{equation}
Thus, it follows that
\begin{displaymath}
	\alpha \alpha^* \alpha = 
 {}_A^X {{\rm id}_X}_A^X {}_A^X f_B^Y  =
 {}_A^X f_B^Y  = 
\alpha
\end{displaymath}
and
\begin{displaymath}
	\alpha^* \alpha \alpha^* = 
 {}_B^Y {{\rm id}_Y}_B^Y   {}_B^Y {f^{-1}}_A^X  = 
 {}_B^Y {f^{-1}}_A^X  = \alpha^*.
\end{displaymath}
Next suppose that 
\begin{equation}\label{aba}
\alpha \beta \alpha = \alpha
\end{equation}
and
\begin{equation}\label{bab}
\beta \alpha \beta = \beta
\end{equation}
where $\beta =  {}_B^{X'} g_A^{Y'} $.
From \eqref{firstset}, \eqref{secondset} and \eqref{aba} it follows that
\begin{align*}
{}_B^Y { f^{-1} }_A^X &= 
\alpha^* = 
\alpha^* \alpha \alpha^* =
\alpha^* \alpha \beta \alpha \alpha^* =
 {}_B^Y {{\rm id}_Y}_B^Y  {}_B^{X'} g_A^{Y'}   {}_A^X {{\rm id}_X}_A^X  \\
&=   {}_B^{ g( g^{-1}(Y \cap X') \cap X ) } 
{ g|_{ g^{-1}(Y \cap X') \cap X } }_A^{ g^{-1}(Y \cap X') \cap X }  .
\end{align*}
Thus we get that $Y = g( g^{-1}(Y \cap X') \cap X ) \subseteq X'$
and $X = g^{-1}(Y \cap X') \cap X \subseteq Y'$. Analogously,
from \eqref{bab}, it follows that $X' \subseteq Y$ and $Y' \subseteq X$.
Thus $X' = Y$ and $Y' = X$ and hence it follows that $\beta = \alpha^*$. 
\end{proof}

\section{Preliminaries on rings}\label{preliminariesrings}

In this section, 
we introduce the category
of partial (commutative) ring isomorphisms
{\rm ISO}$_{cat}$ ({\rm ISOC}$_{cat}$).
This category contains
the well-known
groupoid ISO (ISOC) having all (commutative) rings
as objects and ring isomorphisms as morphisms. 
We also show a result (see Proposition~\ref{isomorphism}), 
concerning central idempotents in rings, that we need in subsequent sections.
Let $A$ be a ring. We always assume that $A$ is
associative and equipped with a multiplicative identity element $1_A$.

\begin{defi}
Let ISO (ISOC) denote the category having all (commutative) rings
as objects and ring isomorphisms as morphisms. 
\end{defi}

\begin{rem}\label{rem:unitalideal}
Recall that an ideal $I$ of a ring $A$ is said to be a \emph{unital ideal}
if $I$, viewed as a ring in itself, is unital.
In this case, the multiplicative identity element of $I$ is denoted by $1_I$
and lies in the center of $A$.
Indeed, let $a \in A$ be arbitrary.
Since $I$ is an ideal of $A$ it follows that
$1_I a , a 1_I \in I$. Thus, $1_I a = 1_I a 1_I = a 1_I$.
Furthermore, if $I$ and $J$ are unital ideals of a ring $A$, then $I \cap J = IJ$.
\end{rem}

Now we will define two ring versions of the inverse category BIJ$_{cat}$
that we defined in the previous section.

\begin{defi}
Let ISO$_{cat}$ (ISOC$_{cat}$) denote the 
subcategory of BIJ$_{cat}$ having (commutative) rings
as objects and, as morphisms, all ${}_A^I f_B^J$
in BIJ$_{cat}$ such that $I$ and $J$ are unital ideals of $A$ and $B$
respectively and $f : J \rightarrow I$ is a ring isomorphism.
Note that the composition of two morphisms,
${}_A^I f_B^J$ and ${}_B^{I'} g_C^{J'}$, in these categories equals
${}_A^{ f(J I') } f|_{J I'} \circ {g|_ {g^{-1}(J I')}} _C^{ g^{-1}(J I') }.$
Define a map $* : ({\rm ISO}_{cat})_1 \rightarrow ({\rm ISO}_{cat})_1$ by
restriction of the map $*$ defined on $( {\rm BIJ}_{cat} )_1$.
This restricts, in turn, to a map
$* : ({\rm ISOC}_{cat})_1 \rightarrow ({\rm ISOC}_{cat})_1$.
\end{defi}

The following is clear.

\begin{prop}\label{ISOISOC}
{\rm ISO}$_{cat}$ and  {\rm ISOC}$_{cat}$ are inverse subcategories of {\rm BIJ}$_{cat}$.
\end{prop}

The \emph{center of $A$}, denoted by $Z(A)$, is the subring of $A$
consisting of all elements $a \in A$ with the property that
for all $b \in A$ the equality $ab = ba$ holds. 
We let ${\rm idem}(A)$ denote
the set of all central idempotents of $A$ and we
let ${\rm ideal}(A)$ denote the set 
of all unital ideals of $A$.

\begin{prop}\label{isomorphism}
Let $A$ be a ring.
The map $\theta : {\rm idem}(A) \rightarrow {\rm ideal}(A)$,
defined by $\theta(x) = Ax$, for $x \in {\rm idem}(A)$,
is an isomorphism of multiplicative monoids.
For all $x \in {\rm idem}(A)$ the equality $Z(Ax) = Z(A)x$ holds.
\end{prop}

\begin{proof}
It is clear that $\theta$ is a homomorphism of multiplicative
monoids. Take $x,y \in {\rm idem}(A)$ such that 
$\theta(x) = \theta(y)$. Then $Ax = Ay$. Since both $x$ and $y$
are multiplicative identity elements for the same monoid, it
follows that $x=y$. Thus $\theta$ is injective.
Now we show that $\theta$ is surjective.
Take $I \in {\rm ideal}(A)$.
Recall that $1_I \in Z(A)$, by Remark~\ref{rem:unitalideal}.
By the idempotency of $1_I$
we get that $\theta(1_I) = A 1_I = I.$ Thus the surjectivity of $\theta$ follows.
For the last part, take $x \in {\rm idem}(A)$.
The inclusion $Z(A)x \subseteq Z( Ax )$ clearly holds.
Take $y \in Z( Ax )$. Then $y = ax$ for some $a \in A$.
Clearly, $yx = ax^2 = ax = y$, since $x$ is idempotent.
Thus, it suffices to show that $y \in Z(A)$.
Take $b \in A$. Then, since $x \in Z(A)$ and $y \in Z(Ax)$, 
we get that $yb = yx b = y bx = yb x = bxy = b y x = by$.
\end{proof}

\section{The Picard Inverse Category}\label{sectionpicardinversecategory}

In this section, we recall the definition of (pre-)equivalence data
and some properties of such systems. 
Then we introduce the Picard inverse category PIC$_{cat}$
(see Definition~\ref{definitionpicardinversecategory} 
and Theorem~\ref{invcat}).
From \cite[Definition (3.2)]{bass1968} we recall the following.

\begin{defi}
A \emph{set of pre-equivalence data} $(I,J,P,Q,\alpha,\beta)$
consists of rings $I$ and $J$,
an $I$-$J$-bimodule $P$,
a $J$-$I$-bimodule $Q$ an $I$-bimodule homomorphism
\begin{equation}\label{ONE}
\alpha : P \otimes_J Q \rightarrow I
\end{equation}
and an $I$-bimodule homomorphism
\begin{equation}\label{TWO}
\beta : Q \otimes_I P \rightarrow J \quad 
\end{equation}
such that the following two diagrams commute
\begin{equation}\label{THREE}
\CD
P \otimes_J Q \otimes_I P @> \alpha \otimes {\rm id}_P >> I \otimes_I P \\
@V {\rm id}_P \otimes \beta VV @VV  V \\
P \otimes_J J @>  >> P \\
\endCD 
\hspace{14mm} 
\CD
Q \otimes_I P \otimes_J Q @> \beta \otimes {\rm id}_Q >> J \otimes_J Q \\
@V {\rm id}_Q \otimes \alpha VV @VV V \\
Q \otimes_I I @>  >> Q \\
\endCD
\end{equation}
where the unlabelled arrows are the multiplication maps.
We shall call it a \emph{set of equivalence data} if $\alpha$ and
$\beta$ are isomorphisms.
\end{defi}

Now we gather some well-known properties concerning pre-equivalence data
that we need in the sequel. 

\begin{prop}\label{proppreequivalence}
If $(I,J,P,Q,\alpha,\beta)$ is a set of pre-equivalence data
such that $\alpha$ (or $\beta$) is surjective, then the following assertions hold:
\begin{itemize}

\item[(a)] $\alpha$ (or $\beta$) is an isomorphism;

\item[(b)] $P$ and $Q$ are generators as $I$-modules (or $J$-modules);

\item[(c)] $P$ and $Q$ are finitely generated and projective
$J$-modules ($I$-modules).

\end{itemize}
\end{prop}

\begin{proof}
See \cite[Theorem (3.4)]{bass1968}.
\end{proof}

\begin{prop}\label{propequivalence}
If $(I,J,P,Q,\alpha,\beta)$ is a set of equivalence data, 
then the ring homomorphisms
\begin{displaymath}
	\End_J(P) \leftarrow I \rightarrow \End_J(Q)^{\rm op}
\end{displaymath}
and
\begin{displaymath}
	\End_I(P)^{\rm op} \leftarrow J \rightarrow \End_I(Q),
\end{displaymath}
induced by the bimodule structure on $P$ and $Q$,
are isomorphisms. These isomorphisms restrict to 
ring isomorphisms
\begin{displaymath}
	\End_{I-J}(P) \leftarrow Z(I) \rightarrow \End_{J-I}(Q)
\end{displaymath}
and
\begin{displaymath}
	\End_{I-J}(P) \leftarrow Z(J) \rightarrow \End_{J-I}(Q),
\end{displaymath}
which in turn restrict to group isomorphisms
\begin{displaymath}
	\Aut_{I-J}(P) \leftarrow U(Z(I)) \rightarrow \Aut_{J-I}(Q)
\end{displaymath}
and
\begin{displaymath}
	\Aut_{I-J}(P) \leftarrow U(Z(J)) \rightarrow \Aut_{J-I}(Q).
\end{displaymath}
\end{prop}

\begin{proof}
See \cite[Theorem (3.5)]{bass1968}.
\end{proof}

\begin{rem}\label{remarkgamma}
Let $(I,J,P,Q,\alpha,\beta)$ be a set of equivalence data. 
It follows from Proposition~\ref{propequivalence} that
there is a unique ring isomorphism 
$\gamma_P : Z(J) \rightarrow Z(I)$ with the property 
that for all $p \in P$ and all $b \in Z(J)$
the equality 
$\gamma_P(b) p = p b$ 
holds. 
\end{rem}

\begin{defi}\label{def:parteqdata}
A \emph{set of partial equivalence data}
\begin{displaymath}
	(A,B,I,J,P,Q,\alpha,\beta)
\end{displaymath}
consists of
rings $A$ and $B$, unital
ideals $I$ and $J$ of, respectively, $A$ and $B$
such that $(I,J,P,Q,\alpha,\beta)$ is a set of equivalence data,
$P$ is an $I$-$J$-bimodule and $Q$ is a $J$-$I$-bimodule.
\end{defi}

\begin{rem}
Note that, with the notation and assumptions of Definition~\ref{def:parteqdata},
$P$ (resp. $Q$) extends uniquely to an $A$-$B$-bimodule (resp. $B$-$A$-bimodule).
Thus, we may interchangeably think of $P$ (resp. $Q$) as an $A$-$B$-bimodule or an $I$-$J$-bimodule (resp. $B$-$A$-bimodule or $J$-$I$-bimodule).
\end{rem}

\begin{prop}\label{composition}
Suppose that 
\begin{equation}\label{FIRST}
(A,B,I,J,P,Q,\alpha,\beta)
\end{equation}
and 
\begin{equation}\label{SECOND}
(B,C,I',J',P',Q',\alpha',\beta')
\end{equation} 
are sets of partial equivalence data.
Then
\begin{displaymath}
	(A, C, I'', J'', P'', Q'', \alpha'', \beta'')
\end{displaymath}
is a set of partial equivalence data, where
\begin{equation}\label{proddata}
I'' = \gamma_P(1_J 1_{I'})A, \,\,
J'' = \gamma_{P'}^{-1}(1_J 1_{I'})C,\,\,
P'' = P \otimes_B P',\,\,
Q'' = Q' \otimes_B Q,\end{equation} and for 
$p \in P$, $p' \in P'$, $q \in Q$, $q' \in Q'$, we put
\begin{displaymath}
	\alpha''(p \otimes p' \otimes q' \otimes q) =
\alpha(p \alpha'(p' \otimes q') \otimes q )
\end{displaymath}
and
\begin{displaymath}
	\beta''(q' \otimes q \otimes p \otimes p') = 
\beta'( q' \otimes \beta(q \otimes p) p' ).
\end{displaymath}
\end{prop}

\begin{proof}
We begin by noticing that $P' \otimes_C Q' \cong P' \otimes_{J'} Q'$.
Indeed, $J'$ is a unital ideal of $C$ and hence $J'=1_{J'}C$.
Moreover, $P'$ is a right $J'$-module and $Q'$ is a left $J'$-module.
Thus, $P' \otimes_C Q' \ni p' \otimes_C q' \mapsto p' \otimes_{J'} q' \in P' \otimes_{J'} Q'$ is a well-defined isomorphism of $I'$-bimodules (and $B$-bimodules).

The map $\alpha'' : P'' \otimes_C Q''  \rightarrow I''$
is an isomorphism of $A$-bimodules since it is the
composition of the following chain of $A$-bimodule isomorphisms:
\begin{align*}
P \otimes_B P' \otimes_C Q' \otimes_B Q & 
\cong
P \otimes_B P' \otimes_{J'} Q' \otimes_B Q
\cong
P \otimes_B I' \otimes_B Q \\
&=
P \otimes_B 1_{I'} B \otimes_B Q
= P 1_{I'} \otimes_B B \otimes_B Q \\
&
\cong
P 1_J 1_{I'} \otimes_B Q \\
&\cong \gamma_P(1_J 1_{I'}) P \otimes_B Q
\cong
\gamma_P(1_J 1_{I'}) I
=
 \gamma_P(1_J 1_{I'}) A.
\end{align*}
Analogously, the map
$\beta'' : Q'' \otimes_A P'' \rightarrow J''$
is an isomorphism of $C$-bimodules since it is the
composition of the following chain of $C$-bimodule isomorphisms:
\begin{align*} Q' \otimes_B Q \otimes_A P \otimes_B P'
& \cong
Q' \otimes_B J \otimes_B P
=
Q' \otimes_B 1_J B \otimes_B P' \\
& =
 Q' \otimes_B B \otimes_B 1_J P'
\cong
Q' \otimes_B 1_J 1_{I'} P' \\
&
\cong
Q' \otimes_B P' \gamma_{P'}^{-1}(1_J 1_{I'})
\cong
J' \gamma_{P'}^{-1}(1_J 1_{I'})
=
 \gamma_{P'}^{-1}(1_J 1_{I'})C.
\end{align*}
From Proposition~\ref{isomorphism} it follows that
$I'' = \gamma_P(1_J 1_{I'})A$ and 
$J'' = \gamma_{P'}^{-1}(1_J 1_{I'})C$
are well-defined.
Now we verify the diagrams in \eqref{THREE}.
By abuse of notation, we let $m$ denote all
of the various multiplication maps.
Take $p_1,p_2 \in P$, $p_1',p_2' \in P'$,
$q_1,q_2 \in Q$ and $q_1',q_2' \in Q'$.
Put
$p_1'' = p_1 \otimes p_1'$, 
$p_2'' = p_2 \otimes p_2''$,
$q_1'' = q_1' \otimes q_1$ 
and
$q_2'' = q_2' \otimes q_2$. 
Then, by making use of \eqref{THREE} for 
\eqref{FIRST} and \eqref{SECOND}, we get that
{\small
\begin{align*}
( m \circ ( \alpha'' \otimes {\rm id}_{P''} ) )
( p_1'' \otimes q_1'' \otimes p_2'' ) &=
 ( m \circ ( \alpha'' \otimes {\rm id}_{P''} ) )
( p_1 \otimes p_1' \otimes q_1' \otimes q_1 \otimes p_2 \otimes p_2' )\\ &
= \alpha( p_1 \alpha'( p_1' \otimes q_1') \otimes q_1 ) p_2 \otimes p_2'\\ &= ( m \circ ( \alpha \otimes {\rm id}_P ) ) 
( p_1 \alpha'(p_1' \otimes q_1') \otimes q_1 \otimes p_2 )  \otimes p_2'\\ &=  ( m \circ ( {\rm id}_P \otimes \beta ) ) 
( p_1 \alpha'(p_1' \otimes q_1') \otimes q_1 \otimes p_2 )  \otimes p_2' \\ &= p_1 \alpha'( p_1' \otimes q_1' ) \beta( q_1 \otimes p_2 )
\otimes p_2' \\ &= p_1 \otimes \alpha'( p_1' \otimes q_1' ) \beta( q_1 \otimes p_2 ) p_2' \\ &= p_1 \otimes 
( m \circ ( \alpha' \otimes {\rm id}_{P'} ) )
(p_1' \otimes q_1' \otimes \beta(q_1 \otimes p_2) p_2' ) \\ &= p_1 \otimes 
( m \circ ( {\rm id}_{P'} \otimes \beta' ) )
(p_1' \otimes q_1' \otimes \beta(q_1 \otimes p_2) p_2' ) \\ &= p_1 \otimes p_1' \beta'( q_1' \otimes \beta( q_1 \otimes p_2 ) p_2' ) \\ &= (m \circ ( {\rm id}_{P''} \otimes \beta'' ) ) 
( p_1'' \otimes q_1'' \otimes p_2'' )
\end{align*}
}
and
{\small
\begin{align*}
( m \circ ( {\rm id}_{Q''} \otimes \alpha'' ) )
( q_1'' \otimes p_1'' \otimes q_2'' ) &= q_1' \otimes q_1 \alpha ( p_1 \alpha' (p_1' \otimes q_2') \otimes q_2 )\\ &= q_1' \otimes ( m \circ ( {\rm id}_Q \otimes \alpha ) )
(q_1 \otimes p_1 \alpha'(p_1' \otimes q_2') \otimes q_2 )\\ &= q_1 \otimes ( m \circ ( \beta \otimes {\rm id}_Q ) )
(q_1 \otimes p_1 \alpha'(p_1' \otimes q_2') \otimes q_2 ) \\ &= q_1' \otimes \beta( q_1 \otimes p_1 \alpha'( p_1' \otimes q_2' ) ) q_2 \\ &= q_1' \otimes \beta(q_1 \otimes p_1) \alpha'(p_1' \otimes q_2') q_2\\ &= q_1' \beta(q_1 \otimes p_1) \alpha'( p_1' \otimes q_2' ) \otimes q_2\\ &= q_1' \alpha'( \beta(q_1 \otimes p_1)p_1' \otimes q_2' ) \otimes q_2\\ &= ( m \circ ( {\rm id}_{Q'} \otimes \alpha' ) )
( q_1' \otimes \beta(q_1 \otimes p_1)p_1' \otimes q_2' ) \otimes q_2 \\ &= (m \circ ( \beta' \otimes {\rm id}_{Q'} ) )
( q_1' \otimes \beta(q_1 \otimes p_1)p_1' \otimes q_2' ) \otimes q_2 \\ &= \beta'( q_1' \otimes \beta( q_1 \otimes p_1 )p_1' ) q_2' \otimes q_2\\ &= ( m \circ ( \beta'' \otimes {\rm id}_{Q''} ) )
(q_1'' \otimes p_1'' \otimes q_2''),
\end{align*} 
}
which finishes the proof.
\end{proof}

To motivate the approach taken later, we now recall the 
definition of the Picard groupoid PIC.

\begin{defi}
Let PIC denote the category 
having as objects all unital rings.
A morphism in PIC from $B$ to $A$ is the collection of all
$A$-$B$-bimodule isomorphism classes $[P]$, for invertible $A$-$B$-bimodules $P$.
Given two such classes $[P]$ and $[Q]$,
where $d([P]) = B = c([Q])$, we put
$[P][Q] = [P \otimes_B Q]$.
Then PIC is a groupoid. Indeed,
if $P$ is an invertible $A$-$B$-bimodule,
then there is an invertible $B$-$A$-bimodule $Q$
such that $P \otimes_B Q \cong A$ (as $A$-bimodules)
and $Q \otimes_A P \cong B$ (as $B$-bimodules).
Thus, if we put $[P]^{-1} = [Q]$, then, clearly
$[P][P]^{-1} = [A]$ and $[P]^{-1}[P] = [B]$.
\end{defi}

\begin{defi}\label{definitionpicardinversecategory}
Let $P$ be an $A$-$B$-bimodule and suppose that
$I$ and $J$ are unital ideals of, respectively, $A$ and $B$,
making $P$ an $I$-$J$-bimodule.
We will indicate this by writing   ${}_A^I P_B^J.$ 
We say that ${}_A^I P_B^J$ is \emph{partially invertible}
if there is ${}_B^J Q_A^I$ and maps $\alpha$ and $\beta$ 
such that $(A,B,I,J,P,Q,\alpha,\beta)$ is a set of partial equivalence data. 
Let PART denote the collection of all partially invertible
bimodules ${}_A^I P_B^J$.
Define an equivalence relation $\sim$ on PART 
by saying that ${}_A^I P_B^J \sim {}_{A'}^{I'} {P'}_{B'}^{J'}$  if 
\begin{displaymath}
	(A,B,I,J) = (A', B', I' J')\,\,\text{ and}\,\,\, 
P \cong P' \,\,\, \text{as \textit{I}-\textit{J}-bimodules.}
\end{displaymath}
The equivalence class of ${}_A^I P_B^J$ in PART 
will be denoted by $[ {}_A^I P_B^J ]$.
The class of objects in ${\rm PIC}_{cat}$ consists of all rings.
The class of morphisms in ${\rm PIC}_{cat}$ consists of all
equivalence classes $[ {}_A^I P_B^J ]$ of partially invertible modules ${}_A^I P_B^J$. 
Define the domain and codomain of a
morphism $[ {}_A^I P_B^J ]$ in ${\rm PIC}_{cat}$ by the relations
$d( [ {}_A^I P_B^J ] ) = B$ and
$c( [ {}_A^I P_B^J ] ) = A$, respectively.
Given a ring $A$, the identity morphism at $A$
is defined to be the morphism $[{}_A^A A_A^A ]$.
Given two morphisms $[ {}_A^I P_B^J ]$ and $[ {}_B^{I'} {P'}_C^{J'} ]$
in ${\rm PIC}_{cat}$ put
$[ {}_A^I P_B^J ] [ {}_B^{I'} {P'}_C^{J'} ] =
[ {}_A^{I''} {P''}_C^{J''} ]$, where $I''$, $P''$ and $J''$
are defined in Proposition~\ref{composition}.
It is clear that the morphisms of the form
$[{}_A^A A_{A}^A]$, for rings $A$, satisfy the
axioms for identity morphisms of ${\rm PIC}_{cat}$.
If $[ {}_A^I P_B^J ] \in ( {\rm PIC}_{cat} )_1$,
then there is ${}_B^J Q_A^I$ and maps $\alpha$ and $\beta$ such that
$(A,B,I,J,P,Q,\alpha,\beta)$ is a set of partial equivalence data. 
Put $[ {}_A^I P_B^J ]^* = [{}_B^J Q_A^I]$.
\end{defi}

\begin{thm}\label{invcat}
${\rm PIC}_{cat}$ is an inverse category.
\end{thm}

\begin{proof}
First we show that the partial composition in $( {\rm PIC}_{cat} )_1$ is associative.
Suppose that
$[ {}_A^{I_1} {P_1}_B^{J_1} ]$,
$[ {}_B^{I_2} {P_2}_C^{J_2} ]$ and
$[ {}_C^{I_3} {P_3}_D^{J_3} ]$
are morphisms in ${\rm PIC}_{cat}$.
We need to show that 
\begin{equation}\label{associativity}
\left(  [ {}_A^{I_1} {P_1}_B^{J_1} ]
[ {}_B^{I_2} {P_2}_C^{J_2} ] \right )
[ {}_C^{I_3} {P_3}_D^{J_3} ]
=
[ {}_A^{I_1} {P_1}_B^{J_1} ]
\left( [ {}_B^{I_2} {P_2}_C^{J_2} ]
[ {}_C^{I_3} {P_3}_D^{J_3} ]   \right ).
\end{equation}
By repeated application of the composition in ${\rm PIC}_{cat}$
it follows that \eqref{associativity} is equivalent to
showing the equalities
\begin{equation}\label{left}
\gamma_{P_1 \otimes_B P_2} ( \gamma_{P_2}^{-1}(1_{J_1} 1_{I_2}) 1_{I_3} ) =
\gamma_{P_1} ( 1_{J_1} \gamma_{P_2} ( 1_{J_2} 1_{I_3} ) )
\end{equation}
and
\begin{equation}\label{right}
\gamma_{P_3}^{-1}( \gamma_{P_2}^{-1} ( 1_{J_1} 1_{I_2} ) 1_{I_3} ) =
\gamma_{P_2 \otimes_C P_3}^{-1} ( 1_{J_1} \gamma_{P_2} ( 1_{J_2} 1_{I_3} ) ).
\end{equation}
First we show \eqref{left}. Take $p_1 \in P_1$ and $p_2 \in P_2$. Then
\begin{align*}p_1 \otimes p_2 \gamma_{P_2}^{-1} ( 1_{J_1} 1_{I_2} ) 1_{I_3}& =
p_1 \otimes p_2 \gamma_{P_2}^{-1} ( 1_{J_1} 1_{I_2} ) 1_{J_2} 1_{I_3} 
= 
p_1 \otimes 1_{J_1} 1_{I_2} p_2 1_{J_2} 1_{I_3} \\&=  p_1 \otimes 1_{J_1} p_2 1_{J_2} 1_{I_3} 
= 
p_1 1_{J_1 }\otimes p_2 1_{J_2} 1_{I_3} \\& =
p_1 1_{J_1 } \otimes \gamma_{P_2} ( 1_{J_2} 1_{I_3} ) p_2 
=  
p_1 1_{J_1 } \gamma_{P_2} ( 1_{J_2} 1_{I_3} ) \otimes p_2 \\& =   
\gamma_{P_1} ( 1_{J_1 } \gamma_{P_2} ( 1_{J_2} 1_{I_3} ) ) p_1 \otimes p_2.
\end{align*}
\noindent Now we show \eqref{right}.  Take $p_2 \in P_2$ and $p_3 \in P_3$. Then
\begin{align*}
1_{J_1} \gamma_{P_2}( 1_{J_2} 1_{I_3} ) p_2 \otimes p_3& =
1_{J_1} 1_{I_2} \gamma_{P_2}( 1_{J_2} 1_{I_3} ) p_2 \otimes p_3 
= 
1_{J_1} 1_{I_2} p_2 1_{J_2} 1_{I_3} \otimes p_3 \\
&= 1_{J_1} 1_{I_2} p_2 1_{I_3} \otimes p_3 
= 
p_2 \gamma_{P_2}^{-1} ( 1_{J_1} 1_{I_2} ) 1_{I_3} \otimes p_3 \\& = 
p_2 \otimes \gamma_{P_2}^{-1} ( 1_{J_1} 1_{I_2} ) 1_{I_3} p_3 
=p_2 \otimes p_3 \gamma_{P_3}^{-1} ( \gamma_{P_2}^{-1} ( 1_{J_1} 1_{I_2} ) 1_{I_3} ),
\end{align*}
as desired.
Next we show the axioms for $*$.
Take $g = [ {}_A^I P_B^J ] \in ( {\rm PIC}_{cat} )_1$.
Then there is ${}_B^J Q_A^I$ and maps $\alpha$ and $\beta$ such that
$(A,B,I,J,P,Q,\alpha,\beta)$ is a set of partial equivalence data. 
Put $g^* = [{}_B^J Q_A^I]$. Then
\begin{align*}g g^* &= [ {}_A^I P_B^J ] [ {}_B^J Q_A^I ]
=\left[ {}_A^{ \gamma_P(1_J 1_J)A } ( P \otimes_B Q )_A^{ \gamma_Q^{-1}(1_J 1_J)A }\right]\\
&= [ {}_A^{ 1_I A } ( P \otimes_B Q )_A^{ 1_I A }]
=[ {}_A^I I_A^I ].
\end{align*}
Using this we get that
\begin{align*}
g g^* g =[ {}_A^I I_A^I ] [ {}_A^I P_B^J ] = 
[ {}_A^{\gamma_I(1_I 1_I)A} (I \otimes_A P)_B^{ \gamma_P^{-1}(1_I 1_I)B } ]
=  [{}_A^{1_I A} P_B^{1_J B} ]
= [ {}_A^I P_B^B ]
= g\end{align*}
and
\begin{align*}
g^* g g^* = 
[ {}_B^J Q_A^I ] [ {}_A^I I_A^I ] = 
\left[ {}_B^{ \gamma_Q(1_I 1_I)B } ( Q \otimes_A I )_A^{ \gamma_I^{-1}(1_I 1_I)A }\right ]
= [ {}_B^{1_J B} Q_A^{1_I A} ] = [ {}_B^J Q_A^I ] = g^*.
\end{align*}
Now we show uniqueness of $g^*$. To this end, first note that
\begin{displaymath}
	g^* g = [ {}_B^J Q_A^I ] [ {}_A^I P_B^J ] =
\left[ {}_B^{ \gamma_Q(1_I 1_I)B } (Q \otimes_A P)_B^{ \gamma_P^{-1}( 1_I 1_I ) B } \right] = [ {}_B^{1_J B} J_B^{1_J B} ] = [ {}_B^J J_B^J ].
\end{displaymath}
Next, suppose that 
\begin{equation}\label{ghg}
ghg = g\,\,\,\,{\rm and}\,\,\,\,
hgh = h
\end{equation}
for some $h = [ {}_B^K M_A^L ]$ with $h^* = [ {}_A^L N_B^K ]$.
From the first equality in \eqref{ghg} it follows that
$g^* g h g g^* = g^* g g^*$ and thus that
$[{}_B^J J_B^J] [ {}_B^K M_A^L ] [ {}_A^I I_A^I ] = [{}_B^J Q_A^I]$.
Rewriting the last equality we get that
\begin{equation}\label{basic}
\left[ {}_B^{ \gamma_J( 1_J \gamma_M(1_L 1_I) )B } 
( J \otimes_B M \otimes_A I)_A^{ \gamma_{M \otimes_A I}^{-1}( 1_J \gamma_M(1_L 1_I) ) A } \right] =
[{}_B^J Q_A^I] 
\end{equation}
and thus that
\begin{equation}\label{inclusion1}
\gamma_J( 1_J \gamma_M(1_L 1_I) )B = J
\end{equation}
and
\begin{equation}\label{inclusion2}
\gamma_{M \otimes_A I}^{-1}( 1_J \gamma_M(1_L 1_I) ) A = I.
\end{equation}
Since $\gamma_J$ is the identity map 
$Z(J) \rightarrow Z(J)$, \eqref{inclusion1} implies that
$1_J \gamma_M(1_L 1_I) B = J$
and hence, in particular, that $J \subseteq K$.
Using that $\gamma_I$ equals the identity map on $Z(I)$,
it follows that
$\gamma_{M \otimes_A I}^{-1} : Z( \gamma_M(1_L 1_I ) ) \rightarrow Z(1_L 1_I)$.
From \eqref{inclusion2} it therefore, in particular,
follows that $I \subseteq L$.
From the second equality in \eqref{ghg} it follows, by symmetry, that
$J \subseteq K$ and $L \subseteq I$.
Thus $J = K$ and $L = I$ and hence
from \eqref{basic} it follows that $h = g^*$.\end{proof}

\begin{exa}[{\bf The Picard semigroup of a commutative ring}]
Let $R$ be a unital commutative ring and let $M$ be a finitely generated (central) $R$-bimodule of rank less than or equal to one, that is, ${\bf rk}(M_\mathfrak p)\leq 1,$ for all $\mathfrak p\in {\rm{Spec}}(R).$ Let $M^*=\Hom_R(M,R)$ be the dual of $M.$ Then by \cite[Proposition 3.8, Lemma 3.9]{DPP} there exists  $e\in \mathrm{idem}(R)$ and an $R$-bimodule isomorphism $\alpha\colon M\otimes_R M^*\to Re,$ given by $\alpha(m\otimes f) := f(m),$ for all $f\in M^* $ and $m\in M.$ Moreover, by \cite[ Lemma 3.10]{DPP} the isomorphisms classes of $M$ and $M^*$ are elements of the Picard group $\Pic(Re).$ In particular, both $M$ and $M^*$ are unital $Re$-bimodules. From this we get that $(R,R,Re,Re,M, M^*, \alpha,\alpha)$ is a set of partial equivalence data. By  \cite[Proposition 3.8]{DPP} the inverse subcategory of ${\rm PIC}_{cat}$, whose only object is $R$ and whose morphisms are of the form $[_{R}^{Re} M_{R}^{Re}]$, is a commutative  inverse semigroup, denoted by ${\rm PicS}(R)$. It was defined in \cite{DPP} and is called \emph{the Picard semigroup of $R.$}
\end{exa}

\begin{defi}\label{defgamma}
Now we will define a partial functor of inverse categories 
$L : {\rm PIC}_{cat} \rightarrow {\rm ISOC}_{cat}$.
If $A$ is a ring, then put $L(A) = Z(A)$.
If $[ {}_A^I P_B^J ]$ is a morphism in PIC$_{cat}$, then put 
$L ( [ {}_A^I P_B^J ] ) = {}_{Z(A)}^{Z(I)} { \gamma_P }_{Z(B)}^{Z(J)} $ 
where $\gamma_P : Z(J) \rightarrow Z(I)$ is the ring isomorphism 
defined in Remark~\ref{remarkgamma}.
\end{defi}

\begin{prop}\label{lfunctor}
The map $L : {\rm PIC}_{cat} \rightarrow {\rm ISOC}_{cat}$ 
is a functor and hence, by Proposition~\ref{ordinaryfunctorpartial},
a partial functor of inverse categories.
\end{prop}

\begin{proof}
Take morphisms $g = [ {}_A^I P_B^J ]$ and $h = [ {}_B^{I'} P_C^{J'} ]$ 
in ${\rm PIC}_{cat}$. Then
\begin{displaymath}
	L(gh) = 
{}_{Z(A)}^{ \gamma_P(1_J 1_{I'}) Z(A) } 
\gamma_{{P \otimes_{B} P'}_{Z(C)}}^{ \gamma_{P'}^{-1}(1_{J} 1_{I'}) Z(C) }
\end{displaymath}
and
\begin{displaymath}
	L(g) L(h) =
{}_{Z(A)}^{ \gamma_P(Z(J)Z(I')) } 
\gamma_P|_{ Z(J)Z(I') } \circ 
\gamma_{P'}|_{\gamma_{P'}^{-1}(Z(J)Z(I') )_{Z(C)}^{\gamma_{P'}^{-1}(Z(J)Z(I'))}}.
\end{displaymath}
Note that
\begin{displaymath}
	\gamma_P( Z(J) Z(I') ) = \gamma_P( 1_J 1_{I'} ) Z(A)
\end{displaymath}
and
\begin{displaymath}
	\gamma_{P'}^{-1}(Z(J)Z(I'))= \gamma_{P'}^{-1}( 1_J 1_{I'} ) Z(C).
\end{displaymath}
Put 
	$\gamma_1 = \gamma_P|_{ Z(J)Z(I') }$
and
	$\gamma_2 = \gamma_{P'}|_{ \gamma_{P'}^{-1}(Z(J)Z(I'))}.$
If $a \in \gamma_{P'}^{-1}(1_J 1_{I'}) Z(C)$, $p \in P$ and $p' \in P'$, then
\begin{displaymath}
	\gamma_{P \otimes_B P'}(a) p \otimes p' =  
p \otimes p' a = 
p \otimes \gamma_2(a) p' =
p \gamma_2(a) \otimes p' =
(\gamma_1 \circ \gamma_2)(a) p \otimes p'.
\end{displaymath}
From Remark~\ref{remarkgamma} it therefore follows that 
$\gamma_{P \otimes_B P'} = \gamma_1 \circ \gamma_2$.
\end{proof}

\section{Epsilon-strongly groupoid graded rings}\label{sectionepsilonstrongly}

In this section, we recall the definition 
of groupoid graded rings and some of their properties.
Then we define epsilon-strongly groupoid graded rings
(see Definition~\ref{defepsilonstronglygradedring})
and provide a characterization of them which 
generalizes an analogous result for group graded rings
(see Proposition~\ref{epsilon}).
Throughout this section, $S$ denotes a ring which is graded by a small groupoid $G$.
Recall from \cite{lundstrom2004,lundstrom2006} that this means that there 
is a set of additive subgroups $\{ S_g \}_{g \in G}$ of
$S$ such that $S = \oplus_{g \in G} S_g$ and, for all $g,h \in G_1$, 
$S_g S_h \subseteq S_{gh}$, if $(g,h) \in G_2$,
and $S_g S_h = \{ 0 \}$, if $(g,h) \notin G_2$.
In that case, $S$ is called \emph{strongly graded} if for all
$(g,h) \in G_2$ the equality $S_g S_h = S_{gh}$ holds.
Given two $G$-graded rings $S$ and $T$, a ring homomorphism 
$f : S \rightarrow T$ is called \emph{graded} if for all $g\in G_1$
the inclusion $f(S_g) \subseteq T_g$ holds.
We put $R = \oplus_{e \in G_0} S_e$.
From the next result it follows that we may
always assume that $G_0$ is finite.

\begin{prop}\label{unit}
With the above notation, we get that $1_S \in R$.
If we put $G_0' = \{ e \in G_0 \mid (1_S)_e \neq 0 \}$ and
$G_1' = \{ g \in G_1 \mid (1_S)_{d(g)} , (1_S)_{c(g)} \neq 0 \}$,
then $G'$ is a subgroupoid of $G$ such that
$G_0'$ is finite and $S = \oplus_{g \in G_1'} S_g$.
\end{prop}

\begin{proof}
This follows from \cite[Proposition 2.1.1]{lundstrom2004}.
\end{proof}

\begin{prop}\label{enough}
The ring $S$ is strongly graded if and only if 
for all $g \in G_1$ the inclusion $1_{S_{c(g)}} \in S_g S_{g^{-1}}$ holds.
\end{prop}

\begin{proof}
This follows from \cite[Lemma 3.2]{lundstrom2006}.
\end{proof}

\begin{defi}\label{defepsilonstronglygradedring}
The ring $S$ is said to be \emph{epsilon-strongly graded by $G$}
if, for each $g \in G_1$, $S_g S_{g^{-1}}$ is a unital ideal of $S_{c(g)}$
such that for all $(g,h) \in G_2$
the equalities $S_g S_h = S_g S_{g^{-1}} S_{gh} = S_{gh} S_{h^{-1}} S_h$ hold.
\end{defi}

\begin{rem}
It follows from Proposition~\ref{unit} and Proposition~\ref{enough} that
if $S$ is strongly graded, then $S$ is epsilon-strongly graded.
\end{rem}

\begin{rem}
Suppose that $S$ is epsilon-strongly graded by $G$ and $g \in G_1$.
Then by the definition of $R,$ the $S_{c(g)}$-ideal 
$S_g S_{g^{-1}}$ is a unital ideal of $R.$ 
Moreover if $\epsilon_g$ is its multiplicative identity element, 
then for $r\in R$ we get that $\epsilon_g r , r \epsilon_g \in S_g S_{g^{-1}}$.
Therefore $\epsilon_g r = ( \epsilon_g r ) \epsilon_g =
\epsilon_g (r \epsilon_g ) = r \epsilon_g,$ which shows that 
$\epsilon_g \in Z(R),$ and $S_g S_{g^{-1}}=\epsilon_g S_{c(g)}=\epsilon_g R$. 
\end{rem}

We now wish to show an epsilon-analogue of Proposition~\ref{enough}.

\begin{prop}\label{epsilon}
The ring $S$ is epsilon-strongly graded by $G$ if and only if for each
$g \in G_1$
there is $\epsilon_g \in S_g S_{g^{-1}}$
such that for each $s \in S_g$ the equalities
$\epsilon_g s = s = s \epsilon_{g^{-1}}$ hold. 
\end{prop}

\begin{proof}
First we show the ``only if'' statement.
Suppose that $S$ is epsilon-strongly graded.
Take $g \in G_1$. Let $\epsilon_g$ denote $1_{ S_g S_{g^{-1}} }$.
Take $s_g \in S_g$.
From Proposition~\ref{unit} it follows that $S_g S_{g^{-1}} S_g = S_g$.
Therefore there is a positive integer $n$ and
$a_g^{(i)} , c_g^{(i)} \in S_g$ and $b_{g^{-1}}^{(i)} \in S_{g^{-1}}$,
for $i \in \{1,\ldots,n\}$, such that
$s_g = \sum_{i=1}^n a_g^{(i)} b_{g^{-1}}^{(i)} c_g^{(i)}$.
Since $\epsilon_g = 1_{ S_g S_{g^{-1}} }$ and
$\epsilon_{g^{-1}} = 1_{ S_{g^{-1}} S_g }$, we get that
\begin{displaymath}
	\epsilon_g s_g = 
\sum_{i=1}^n ( \epsilon_g a_g^{(i)} b_{g^{-1}}^{(i)} ) c_g^{(i)} =
\sum_{i=1}^n a_g^{(i)} b_{g^{-1}}^{(i)} c_g^{(i)} = s_g
\end{displaymath}
and
\begin{displaymath}
	s_g \epsilon_{g^{-1}} = 
\sum_{i=1}^n a_g^{(i)} ( b_{g^{-1}}^{(i)} c_g^{(i)} \epsilon_{g^{-1}} ) =
\sum_{i=1}^n a_g^{(i)} b_{g^{-1}}^{(i)} c_g^{(i)} = s_g.
\end{displaymath}
Now we show the  ``if'' statement.
Suppose that to each
$g \in G_1$ there is $\epsilon_g \in S_g S_{g^{-1}}$
such that for each $s \in S_g$ the equalities
$\epsilon_g s = s = s \epsilon_{g^{-1}}$ hold. 
Take $(g,h) \in G_2$. Then, from Proposition~\ref{unit}, it follows that
\begin{displaymath}
	S_g S_h = 
\epsilon_g S_g S_h \subseteq 
S_g S_{g^{-1}} S_g S_h \subseteq 
S_g S_{g^{-1}} S_{gh} \subseteq 
S_g S_{d(g)h}=
S_g S_{c(h)h}=S_g S_{h}
\end{displaymath}
and
\begin{displaymath}
	S_g S_h = 
S_g S_h \epsilon_{ h^{-1} } \subseteq 
S_g S_h S_{h^{-1}} S_h \subseteq 
S_{gh}  S_{h^{-1}} S_h =
S_{gc(h)} S_h =
S_{gd(g)} S_h=
S_{g} S_h.
\end{displaymath}
Moreover, it is clear that $\epsilon_g$ is the multiplicative identity element of $ S_g S_{g^{-1}}.$
\end{proof}

\section{Examples of epsilon-strongly groupoid graded rings}\label{Sec:ExEpsGraded}

In this section we present some examples of epsilon-strongly groupoid graded rings.

\subsection{Partial skew groupoid rings}

The notion of a partial action of a groupoid on a ring, as well as the construction of the corresponding partial skew groupoid ring, is due to Bagio and Paques \cite{bagio2012}.

Let $G$ be a groupoid and suppose that  
$B$ is a ring which is the product of a collection of rings $\{ B_e \}_{e \in G_0}$.

\begin{defi}\label{defipartialaction}
A \emph{partial groupoid action of $G$ on $B$} is a collection
of maps $\{ \theta_g \}_{g \in G_1}$, 
where, for each $g \in G_1$, $B_g$ is an ideal of $B_{c(g)},$ $B_{c(g)}$ is an ideal of $B$ and
$\theta_g : B_{g^{-1}} \rightarrow B_g$ is a ring isomorphism
satisfying the following three axioms:
\begin{itemize}

\item[(G1)] if $e \in G_0$, then $\theta_e = {\rm id}_{B_e}$;

\item[(G2)] if $(g,h) \in G_2$, then 
$\theta\m_h ( B_{g^{-1}} \cap B_h ) = B_{(gh)\m}$;

\item[(G3)] if $(g,h) \in G_2$ and $x \in \theta\m_h ( B_{g^{-1}} \cap B_h )$,
then $\theta_g( \theta_h ( x ) ) = \theta_{gh} (x)$.

\end{itemize}
Note that conditions (G2) and (G3) are equivalent to the 
fact that $\theta_{gh}$ is an extension of $\theta_g \circ \theta_h.$ We say that $\theta$ is \emph{global} if  $\theta_{gh}=\theta_g \circ \theta_h,$ for each $(g,h) \in G_2.$
\end{defi}

\begin{defi}\label{defiunitalpartialaction}
Let $\{ \theta_g \}_{g \in G_1}$
be a partial groupoid action of $G$ on $B$.
Suppose that for each 
each $g \in G_1$, $B_g$ is unital, i.e. $B_g$ is generated 
by an idempotent $1_g$ which is central in $B_{c(g)},$ and $\theta_g$ is a monoid isomorphism.
In that case, we say that  $\{ \theta_g \}_{g \in G_1}$
is a \emph{unital  partial groupoid action of $G$ on $B$}.
\end{defi}

The {\it partial skew groupoid ring} $B \star_\theta G$,
associated with a unital partial groupoid action $\{ \theta_g \}_{g \in G_1}$
of $G$ on $B$,
is the set
of all finite formal sums $\sum_{g \in G_1} b_g \delta_g$, 
where $b_g \in B_g$, with addition defined componentwise and  
multiplication determined by the rule 
\begin{equation}\label{multiplicationrule}
(b_g \delta_g) (b_h' \delta_h) = b_g\alpha_g(  b_h' 1_{g^{-1}} ) \delta_{gh},
\end{equation}
if $(g,h) \in G^2$, and $(r_g \delta_g) (r_h' \delta_h) = 0$, otherwise.

There is a natural $G$-grading on $B \star_\theta G$.
Indeed, if we put $S_g = B_g \delta_g$ for each $g\in G_1$,
then $B \star_\theta G = \oplus_{g\in G_1} S_g$ is $G$-graded.
For each $g\in G$ the idempotent $1_g\delta_{c(g)}$ satisfies the conditions of  Proposition~\ref{epsilon}. Thus, $B \star_\theta G$ is an epsilon-strongly $G$-graded ring. Moreover, by \cite[Proposition 2.5]{BFP} one has that  $B \star_\theta G$ is strongly $G$-graded, if and only if, $\theta$ is global.

\subsection{Leavitt path algebras}


Let $E=(E^0,E^1,r,s)$ be a directed graph, consisting of two countable sets $E^0$, $E^1$ and maps $r,s : E^1 \to E^0$. The elements of $E^0$ are called \emph{vertices} and the elements of $E^1$ are called \emph{edges}. If both $E^0$ and $E^1$ are finite sets, then we say that $E$ is \emph{finite}.
A \emph{path} $\mu$ in 
$E$ is a sequence of edges $\mu = \mu_1 \ldots \mu_n$ such that $r(\mu_i)=s(\mu_{i+1})$ for $i\in \{1,\ldots,n-1\}$. In such a case, $s(\mu):=s(\mu_1)$ is the \emph{source} of $\mu$, $r(\mu):=r(\mu_n)$ is the \emph{range} of $\mu$ and $n$ is the \emph{length} of $\mu$.

\begin{defi}[\cite{hazrat2013}]
Let $E$ be any directed graph and let $K$ be a unital ring.
The \emph{Leavitt path algebra of $E$ with coefficients in $K$}, denoted by $L_K(E)$, is the
algebra generated by a set $\{v \mid v\in E^0\}$ of pairwise orthogonal idempotents, together with a set of elements $\{f \mid f\in E^1\} \cup
\{f^* \mid f\in E^1\}$, which satisfy the following  relations:
\begin{enumerate}
\item $s(f)f=fr(f)=f$, for all $f\in E^1$;
\item $r(f)f^*=f^*s(f)=f^*$, for all $f\in E^1$;
\item $f^*f'=\delta_{f,f'}r(f)$, for all $f,f'\in E^1$;
\item $v=\sum_{ \{ f\in E^1 \mid s(f)=v \} } ff^*$, for every $v\in E^0$ for which $s^{-1}(v)$ is non-empty and finite.
\end{enumerate}
Here the ring $K$ commutes with the generators.
\end{defi}

\begin{rem}
(a) Every path $\mu = \mu_1 \ldots \mu_n$ may be viewed as an element of $L_K(E)$.
Given such an element $\mu$, we put $\mu^* := \mu_n^* \ldots \mu_1^* \in L_K(E)$.
The element $\mu^*$ may also be thought of as a \emph{ghost path} in $E$,
as opposed to $\mu$ which is a \emph{real path}.

(b) Note that every element $x\in L_K(E)$ may be written on the form
$x=\sum_{i=1}^n k_i \alpha_i \beta_i^*$, for suitable $k_i \in K$ and (real) paths $\alpha_i$ and $\beta_i$ satisfying $r(\alpha_i)=r(\beta_i)$, for $i\in \{1,\ldots,n\}$.
\end{rem}

\subsubsection{The groupoid}

Based on $E$, we define a groupoid $G$ in the following way.
The objects of $G$ are the vertices of $E$, i.e. $G_0=E^0$.
An ordered pair of vertices, $(u,v)$, is an arrow in $G$
with $d(u,v)=v$ and $c(u,v)=u$,
if there is a path
\begin{displaymath}
\mu=\mu_1 \mu_2 \ldots \mu_n
\end{displaymath}
such that $u=s(\mu)=s(\mu_1)$ and $v=r(\mu)=r(\mu_n)$,
where $\mu_i \in E^1 \cup (E^1)^* \cup E^0$ for each $i\in \{1,2,\ldots,n\}$,
and $r(\mu_i)=s(\mu_{i+1})$ for each $i\in \{1,2,\ldots,n-1\}$.

Two arrows $(u,v)$ and $(v',w)$ and composable if only if $v'=v$.
In that case, their composition is defined to be equal to
\begin{displaymath}
	(u,v)(v,w)=(u,w).
\end{displaymath}

\subsubsection{The grading}
Let $W$ denote the set of finite real paths in $E$,
and consider $W$ as a subset of the ring $L_K(E)$.

\begin{lem}
If we, for each $(u,v) \in G_1$, put
\begin{displaymath}
	S_{(u,v)} = \mathrm{span}_K \{ \alpha\beta^* \mid \alpha,\beta \in W \text{ such that } s(\alpha)=u, r(\alpha)=r(\beta), s(\beta)=v \}
\end{displaymath}
then this turns $S = L_K(E) = \oplus_{(u,v)\in G_1} S_{(u,v)}$ into a $G$-graded ring.
\end{lem}

\begin{proof}
Clearly, $L_K(E) = \sum_{(u,v) \in G_1} S_{(u,v)}$
and this sum is in fact direct.
Indeed, take a non-zero $x \in S_{(u,v)} \cap S_{(a,b)}$.
Then $x=ux$ and $x=ax$, and hence
$0 \neq x = ux = u(ax)=(ua)x$. In particular, $ua\neq 0$ which means that $u=a$.
Similarly, we may conclude that $v=b$. That is, $(u,v)=(a,b)$.

Let $(u,v),(v',w)\in G_1$ be arbitrary.
If $v'=v$ then we get that
$S_{(u,v)} S_{(v,w)} \subseteq S_{(u,w)}$.
On the other hand, if $v'\neq v$, then $S_{(u,v)} S_{(v',w)} = \{0\}$.
This shows that $L_K(E)$ is $G$-graded.
\end{proof}

\begin{thm}\label{thm:epsgradLPA}
If $E$ is a finite graph,
then the Leavitt path algebra $S=L_K(E)$ is epsilon-strongly $G$-graded.
\end{thm}

\begin{proof}
Let $(u,v) \in G_1$ be arbitrary.

If $u \notin S_{(u,v)} S_{(v,u)}$,
then we shall be interested in the following set:
\begin{displaymath}
	P_{(u,v)} = \{ \alpha \mid \alpha \beta^* \in S_{(u,v)} \}.
\end{displaymath}
For $\alpha_i, \alpha_j \in P_{(u,v)}$, we write $\alpha_i \leq \alpha_j$ if $\alpha_i$ is an initial subpath of $\alpha_j$.
Clearly, $\leq$ is a partial order on $P_{(u,v)}$.
Moreover, using that $E$ is finite, it is not difficult to see that there can only be a finite number of minimal elements of $P_{(u,v)}$ with respect to $\leq$.
We collect all such minimal elements in the set $M_{(u,v)} = \{\alpha_1,\ldots,\alpha_k\}$.

We are now ready to define $\epsilon_{(u,v)}$ in the following way:

\begin{displaymath}
	\epsilon_{(u,v)} = \left\{
	\begin{array}{ll}
		u & \text{if } u \in S_{(u,v)} S_{(v,u)}\\
		\sum_{\alpha_j \in M_{(u,v)}} \alpha_j \alpha_j^* & \text{otherwise}
	\end{array}
	\right.
\end{displaymath}
Note that, whenever $\alpha \beta^* \in S_{(u,v)}$ is a non-zero monomial, i.e. $r(\alpha)=r(\beta)$,
we get that
$\alpha \alpha^* = \alpha r(\beta) \alpha^* = \alpha \beta^* \beta \alpha^* \in S_{(u,v)} S_{(v,u)}$.
In particular, $\alpha_j \alpha_j^* \in S_{(u,v)} S_{(v,u)}$ for each $\alpha_j \in M_{(u,v)}$.
Hence, by construction, $\epsilon_{(u,v)} \in S_{(u,v)} S_{(v,u)}$.
Moreover, $(\epsilon_{(u,v)})^* = \epsilon_{(u,v)}$.

Take any monomial $\gamma \delta^* \in S_{(u,v)}$.
First we show that $\epsilon_{(u,v)} \gamma \delta^* = \gamma \delta^*$.

\noindent \textbf{Case 1:} ($u \in S_{(u,v)} S_{(v,u)}$)

Clearly,
$\epsilon_{(u,v)} \gamma \delta^* = u \gamma \delta^* = \gamma \delta^*$.

\noindent \textbf{Case 2:} ($u \notin S_{(u,v)} S_{(v,u)}$)

Note that there is some $\alpha' \in M_{(u,v)}$ such that $\gamma = \alpha'\gamma'$.
Thus,
\begin{align*}
	\epsilon_{(u,v)} \gamma \delta^* &=  \left(\alpha'\alpha'^*+ \sum_{\alpha_j \in M_{(u,v)} \setminus \{\alpha'\}} \alpha_j \alpha_j^* \right) \gamma \delta^* \\
	&= \alpha' \alpha'^* \alpha' \gamma' \delta^* + 0 = \alpha' \gamma' \delta^* = \gamma \delta^*.
\end{align*}

\noindent It remains to show that
$\gamma \delta^* \epsilon_{(v,u)} = \gamma \delta^*$.
Note that
\begin{displaymath}
	\gamma \delta^* \in S_{(u,v)} \Longleftrightarrow \delta \gamma^* \in S_{(v,u)}.
\end{displaymath}
It follows, from Case 1 and Case 2, that
$\epsilon_{(v,u)} \delta \gamma^* = \delta \gamma^*$.
Using this we get that
\begin{displaymath}
	\gamma \delta^* \epsilon_{(v,u)} =
	\gamma \delta^* (\epsilon_{(v,u)})^* =
	(\epsilon_{(v,u)} \delta \gamma^*)^* =
	(\delta \gamma^*)^*
	= \gamma \delta^*.
\end{displaymath}
This concludes the proof.
\end{proof}

In general $L_K(E)$ need not be strongly $G$-graded, as the following example shows.

\begin{exa}
Let $K$ be a unital ring
and consider the Leavitt path algebra $L_K(E)$
associated with the following graph $E$:
\begin{displaymath}
	\xymatrix{
	\bullet_{v_1} & \ar[l]_{f_1} \bullet_{v_2}  \ar[r]^{f_2} & \bullet_{v_3}
	}
\end{displaymath}
A few short calculations reveal that
\begin{itemize}
	\item $S_{(v_2,v_1)} S_{(v_1,v_2)} = K f_1 f_1^*$
	\item $S_{(v_1,v_2)} S_{(v_2,v_1)} = K v_1$
	\item $S_{(v_3,v_2)} S_{(v_2,v_3)} = K v_3$
	\item $S_{(v_2,v_3)} S_{(v_3,v_2)} = K f_2 f_2^*$
	\item $S_{(v_3,v_1)} S_{(v_1,v_3)} = \{0\}$
	\item $S_{(v_1,v_3)} S_{(v_3,v_1)} = \{0\}$
\end{itemize}
and we may choose
\begin{itemize}
	\item $\epsilon_{(v_2,v_1)} = f_1 f_1^*$
	\item $\epsilon_{(v_1,v_2)} = v_1$
	\item $\epsilon_{(v_3,v_2)} = v_3$
	\item $\epsilon_{(v_2,v_3)} = f_2 f_2^*$
	\item $\epsilon_{(v_1,v_3)} = \epsilon_{(v_3,v_1)} = 0$.
\end{itemize}
Clearly, 
$(v_1,v_3)$ and $(v_3,v_1)$ are composable,
but
$\{0\} = S_{(v_1,v_3)} S_{(v_3,v_1)} \neq S_{(v_1,v_1)}$.
Thus, $L_K(E)$ is not strongly $G$-graded.
\end{exa}

\begin{rem}
Gon\c{c}alves and Yoneda \cite{GoncalvesYoneda2016}
have shown that each Leavitt path algebra may be viewed as
a partial skew groupoid ring.
Their observation gives rise to another example of
an epsilon-strong groupoid grading on a Leavitt path algebra.
\end{rem}

\subsection{Morita rings}

Let $(A,B, _AM_B, _BN_A, \varphi, \phi)$ be a strict Morita context.
It consists of unital rings $A$ and $B$,
an $A$-$B$-bimodule $M$, a $B$-$A$-bimodule $N$,
an $A$-$A$-bimodule epimorphism $\varphi : M \otimes_B N \to A$
and a $B$-$B$-bimodule epimorphism $\phi : N \otimes_A M \to B$.

The associated \emph{Morita ring} 
is the set
\begin{displaymath}
	S =
	\left(
	\begin{array}{cc}
		A & M \\
		N & B
	\end{array}
	\right)
\end{displaymath}
 with the natural addition and with a multiplication defined by
\begin{displaymath}
		\left(
	\begin{array}{cc}
		a_1 & m_1 \\
		n_1 & b_1
	\end{array}
	\right)
	*
		\left(
	\begin{array}{cc}
		a_2 & m_2 \\
		n_2 & b_2
	\end{array}
	\right)
	=
		\left(
	\begin{array}{cc}
		a_1a_2 + \varphi(m_1 \otimes n_2) & a_1m_2 + m_1 b_2\\
		n_1a_2 + b_1n_2 & \phi(n_1 \otimes m_2) + b_1b_2
	\end{array}
	\right)
\end{displaymath}
for $a_1,a_2\in A$, $b_1,b_2\in B$, $m_1,m_2\in M$ and $n_1,n_2\in N$.
Let $G$ be a group and $I$ a non-empty set. Then the set $I\times G\times I$ considered as morphisms, where the composition is given by the rule 
\begin{displaymath}
	(i,g,j)(j,h,k)=(i, gh,k),
\end{displaymath}
for all $i,j,k\in I$ and $g,h\in G$, is a groupoid. Using this groupoid and taking $I=\{1,2\}$ and $G$ the infinite cyclic group generated by $g$
we can define a grading on $S$ by putting
\begin{displaymath}
	S_{(1,e, 1)} =
	\left(
	\begin{array}{cc}
		A & 0 \\
		0 & 0
	\end{array}
	\right),
	\quad
S_{(2,e, 2)} =
\left(
	\begin{array}{cc}
		0 & 0 \\
		0 & B
	\end{array}
	\right),
\end{displaymath}\begin{displaymath}
		S_{(1,g, 2)} =
	\left(
	\begin{array}{cc}
		0 & M \\
		0 & 0
	\end{array}
	\right),
\quad	
		S_{(2,g\m, 1)} =
	\left(
	\begin{array}{cc}
		0 & 0 \\
		N & 0
	\end{array}
	\right)
\end{displaymath}
and $S_{(i,h, j)}=\left\{\left(\begin{smallmatrix}0&0\\0&0\end{smallmatrix}\right)\right\},$ in any other case.
 Then for $h \in G \setminus \{e,g,g^{-1}\}$ we have that 
\begin{displaymath}
	S_{(1,h\m, 1)} S_{(1,h, 1)}\neq S_{(1,e,1)}
\end{displaymath}
and thus $S$ is not strongly graded.
However, $S$ is epsilon-strongly graded. Indeed, it is easy to see that
\begin{displaymath}
	S_{(1,g, 2)} S_{(2,g\m, 1)}
	=
		\left(
	\begin{array}{cc}
		\image(\varphi) & 0 \\
		0 & 0
	\end{array}
	\right) =\left(
	\begin{array}{cc}
		A & 0 \\
		0 & 0
	\end{array}
	\right)
\end{displaymath}
and
\begin{displaymath}
S_{(2,g\m, 1)} S_{(1,g, 2)}
	=
		\left(
	\begin{array}{cc}
		0 & 0 \\
		0 & \image(\phi) 
	\end{array}
	\right)	=\left(
	\begin{array}{cc}
		0  & 0 \\
		0 & B
	\end{array}
	\right).
\end{displaymath}
If we put
\begin{displaymath}
\epsilon_{(1,e,1)}=	\epsilon_{(1,g,2)}
	=
		\left(
	\begin{array}{cc}
		1_A & 0 \\
		0 & 0
	\end{array}
	\right)
	\quad
	\text{ and }
	\quad
	\epsilon_{(2,e,2)}=\epsilon_{(2,g\m,1)} 
	=
		\left(
	\begin{array}{cc}
		0 & 0 \\
		0 & 1_B
	\end{array}
	\right),	
\end{displaymath}
then, by Proposition \ref{epsilon} this yields an epsilon-strong $(I\times G\times I)$-grading on $S$.

\section{Epsilon-strongly groupoid graded modules}\label{sectionepsilonmodules}

In this section,
we define epsilon-strongly groupoid  graded modules
(see Definition~\ref{definitionepsilonmodules}) and we 
provide a characterization of them 
(see Proposition~\ref{tensorg}) that generalizes
a result \cite[Theorem I.3.4]{nas1982} previously obtained for strongly group graded modules.
At the end of this section, we show (see Proposition~\ref{tensor})
that the multiplication maps
$m_{g,h} : S_g \otimes_R S_h \rightarrow \epsilon_g S_{gh} = S_{gh} \epsilon_{h^{-1}}$,
for $(g,h) \in G_2$, are $R$-bimodule isomorphisms.
In particular this implies that for every $g \in G_1$, the sextuple
\begin{displaymath}
	( \epsilon_g R , \epsilon_{g^{-1}}R , S_g , S_{g^{-1}} , m_{g,g^{-1}} , m_{g^{-1},g} )
\end{displaymath}
is a set of equivalence data.
Throughout this section, $S$ denotes a ring 
which is graded by a small groupoid $G$,
and we put $R = \oplus_{e \in G_0} S_e$.
For the entirety of this section also let 
$M$ be a graded left (right) $S$-module.
Recall that this means that there 
to each $g \in G_1$ is an additive subgroup $M_g$ of $M$
such that $M = \oplus_{g \in G_1} M_g$, as additive groups,
and for all $g,h \in G_1$, the inclusion
$S_g M_h \subseteq M_{gh}$ 
(or $M_g S_h \subseteq M_{gh}$)
holds, if $(g,h) \in G_2$, and
$S_g M_h = \{ 0 \}$
(or $M_g S_h = \{ 0 \}$), otherwise.
Recall that $M$ is called \emph{strongly graded} if for all $(g,h) \in G_2$
the equality $S_g M_h = M_{gh}$ (or $M_g S_h = M_{gh}$) holds. 

\begin{defi}\label{definitionepsilonmodules}
We say that $M$ is \emph{epsilon-strongly graded} if, for each 
$g \in G_1$, $S_g S_{g^{-1}}$ is a unital ideal of $S_{c(g)}$
such that for all $(g,h) \in G_2$ the equality
$S_g M_h = S_g S_{g^{-1}} M_{gh}$ 
($M_g S_h = M_{gh} S_{h^{-1}} S_h$) holds.
\end{defi}

\begin{prop}\label{tensorg}
The following assertions are equivalent:
\begin{itemize}

\item[(a)] The ring $S$ is epsilon-strongly graded;

\item[(b)] Every graded left $S$-module is 
epsilon-strongly graded;

\item[(c)] Every graded right $S$-module is
epsilon-strongly graded.

\end{itemize}
\end{prop}

\begin{proof}
Suppose that (a) holds. First we show that (b) holds.
Let $M$ be a $G$-graded left $S$-module
and take $(g,h) \in G_2$. Then
\begin{displaymath}
	S_g S_{g^{-1}} M_{gh} \subseteq
S_g M_{g^{-1} g h} = S_g M_h = S_g S_{g^{-1}} S_g M_h\subseteq S_g S_{g^{-1}} M_{gh}.
\end{displaymath}
Next we show that (c) holds.
Let $M$ be a $G$-graded right $S$-module
and take $(g,h) \in G_2$. Then
\begin{displaymath}
	M_{gh} S_{h^{-1}} S_h \subseteq
 M_g S_h = M_g   S_hS_{h^{-1}}S_h 
\subseteq M_{gh} S_{h^{-1}}S_h.
\end{displaymath}
It is clear that (b) (or (c)) implies (a). 
\end{proof}

\begin{prop}\label{tensor}
Suppose that $S$ is epsilon-strongly graded, 
and let $\{\epsilon_g\}_{g \in G_1}$ be the family of 
central idempotents of $R$ provided by Proposition~\ref{epsilon}. 
Then for  all $(g,h) \in G_2$ the following assertions hold:
\begin{itemize}

\item[(a)] For every graded left $S$-module $M$
 the multiplication map
$m_{g,h} : S_g \otimes_{S_{d(g)}} M_h \rightarrow \epsilon_g M_{gh}$
is an isomorphism of $R$-bimodules;

\item[(b)] For every graded right $S$-module $M$
 the multiplication map
$m_{g,h}' : M_g \otimes_{S_{d(g)}} S_h \rightarrow M_{gh} \epsilon_{h^{-1}}$
is an isomorphism of $R$-bimodules;

\item[(c)] The multiplication map
$m_{g,h} : S_g \otimes_R S_h \rightarrow \epsilon_g S_{gh} = S_{gh} \epsilon_{h^{-1}}$
is an isomorphism of $R$-bimodules;

\item[(d)] For every $g \in G_1$, the sextuple 
\begin{displaymath}
	( \epsilon_g R , \epsilon_{g^{-1}}R , S_g , S_{g^{-1}} , m_{g,g^{-1}} , m_{g^{-1},g} )
\end{displaymath}
is a set of equivalence data.

\end{itemize}
\end{prop}

\begin{proof} Take $(g,h) \in G_2$.

(a) Let $M$ be a $G$-graded left $S$-module.
From Proposition~\ref{tensorg}(ii) it follows that $m_{g,h}$ is surjective.
Now we show that $m_{g,h}$ is injective.
To this end, take a positive integer $n$ and 
$s_g^{(i)} \in S_g$ and $l_h^{(i)} \in M_h$,
for $i \in \{1,\ldots,n\}$, such that $m_{g,h}(x) = 0$, 
where $x = \sum_{i=1}^n s_g^{(i)} \otimes l_h^{(i)} \in S_g \otimes_{S_{d(g)}} M_h$.
Take a positive integer $m$ and $u_g^{(j)} \in S_g$ and
$v_{g^{-1}}^{(j)} \in S_{g^{-1}} $,
for $j \in \{1,\ldots,m\}$, such that 
$\epsilon_g = \sum_{j=1}^m u_g^{(j)} v_{g^{-1}}^{(j)}$. Then 
\begin{align*}
x & = 
\sum_{i=1}^n s_g^{(i)} \otimes l_h^{(i)} 
= \sum_{i=1}^n \epsilon_g s_g^{(i)} \otimes l_h^{(i)} 
= \sum_{i=1}^n \sum_{j=1}^m u_g^{(j)} v_{g^{-1}}^{(j)} s_g^{(i)} \otimes l_h^{(i)} \\
& = \sum_{i=1}^n \sum_{j=1}^m u_g^{(j)} \otimes v_{g^{-1}}^{(j)} s_g^{(i)} l_h^{(i)} 
= \sum_{j=1}^m u_g^{(j)} \otimes v_{g^{-1}}^{(j)} m_{g,h}(x) = 0.
\end{align*}

(b) Let $M$ be a $G$-graded right $S$-module.
From Proposition~\ref{tensorg}(iii) it follows that $m_{g,h}'$ is surjective.
Now we show that $m_{g,h}'$ is injective.
To this end, take a positive integer $n$ and 
$m_g^{(i)} \in M_g$ and $s_h^{(i)} \in S_h$,
for $i \in \{1,\ldots,n\}$, such that $m_{g,h}'(x) = 0$, 
where $x = \sum_{i=1}^n l_g^{(i)} \otimes s_h^{(i)} \in M_g \otimes_{S_{d(g)}} S_h$.
Take a positive integer $m$, and $u_{h^{-1}}^{(j)} \in S_{h^{-1}}$ and
$v_{h}^{(j)} \in S_{h}$,
for $j \in \{1,\ldots,m\}$, such that 
$\epsilon_{h^{-1}} = \sum_{j=1}^m u_{h^{-1}}^{(j)} v_{h}^{(j)}$. Then 
\begin{align*}
x & = 
\sum_{i=1}^n l_g^{(i)} \otimes s_h^{(i)} 
= \sum_{i=1}^n l_g^{(i)} \otimes s_h^{(i)} \epsilon_{h^{-1}}
= \sum_{i=1}^n \sum_{j=1}^m l_g^{(i)} \otimes s_h^{(i)} u_{h^{-1}}^{(j)} v_{h}^{(j)} \\
&= \sum_{i=1}^n \sum_{j=1}^m l_g^{(i)} s_h^{(i)} u_{h^{-1}}^{(j)} \otimes v_{h}^{(j)}
= \sum_{j=1}^m m_{g,h}'(x) u_{h^{-1}}^{(j)} \otimes v_{h}^{(j)} = 0.\end{align*}

(c) and (d) follow immediately from (a) or (b).
\end{proof}

\begin{rem}
Take $g \in G_1$.
It is clear from the definition of epsilon-strongly groupoid graded rings that the sextuplet 
\begin{displaymath}
	( \epsilon_g R , \epsilon_{g^{-1}}R , S_g , S_{g^{-1}} , m_{g,g^{-1}} , m_{g^{-1},g} )
\end{displaymath}
is a set of pre-equivalence data with $m_{g,g^{-1}}$ and $m_{g^{-1},g}$ surjective.
Thus, injectivity of the maps $m_{g,g^{-1}}$ and $m_{g^{-1},g}$ also follow from
Proposition~\ref{proppreequivalence}(a). 
\end{rem}

\section{Generalized Epsilon-crossed products}\label{sectionepsiloncrossedproducts}

In this section,
we introduce generalized epsilon-crossed groupoid products
(see Definition~\ref{defpartialgeneralizedepsiloncrossedproduct})
and we show that they parametrize the family of 
epsilon-strongly groupoid graded rings
(see Proposition~\ref{firstcorrespondence} and Proposition~\ref{secondcorrespondence}).
Throughout this section $G$ denotes a small groupoid with $G_0$ finite.

\begin{defi}
Suppose that $F : G \rightarrow {\rm PIC}_{cat}$ is a partial functor
of inverse categories. For each $e \in G_0$ define the ring $A_e$ by $F(e) = A_e$. 
For each $g \in G_1$ put
$F(g) = \left[ {}_{A_{c(g)}}^{I_g} (P_g)_{A_{d(g)}}^{J_g} \right]$
for some $A_{c(g)}$-$A_{d(g)}$-bimodule $P_g$,
some unital ideal $I_g$ of $A_{c(g)}$ and some unital ideal $J_g$ of $A_{d(g)},$ 
making $P_g$ an $I_g$-$J_g$-bimodule.
For the time being, assume that the bimodules $P_g$, for $g \in G_1$, are fixed.
From the equality $F(g^{-1}) = F(g)^*$ it follows by the proof of Theorem~\ref{invcat} that
$J_g = I_{g^{-1}}$ so we may write
$F(g) =\left [ {}_{A_{c(g)}}^{I_g} (P_g)_{A_{d(g)}}^{I_{g^{-1}}} \right]$.
For each $g \in G_1$, put $\epsilon_g = 1_{I_g}$.
\end{defi}

\begin{prop}\label{paction}
Suppose that $(g,h) \in G_2$. Then 
$\gamma_{P_{gh}} (\epsilon_{ (gh)^{-1} }\epsilon_{h^{-1}}) =
\epsilon_g \epsilon_{gh}$.
In particular, $\epsilon_g P_{gh} = P_{gh} \epsilon_{h^{-1}}$.
\end{prop}

\begin{proof}
From  \eqref{proddata} it follows that
\begin{align*}F(g) F(g^{-1}) F(gh) &= 
\left[ {}_{A_{c(g)}}^{I_g} (P_g)_{A_{d(g)}}^{I_{g^{-1}}} \right]
\left[  {}_{A_{d(g)}}^{ I_{g^{-1}} } ( P_{g^{-1}} )_{A_{c(g)}}^{I_g} \right]
\left[ {}_{A_{c(g)}}^{I_{gh}} ( P_{gh} )_{A_{d(h)}}^{I_{(gh)^{-1}}} \right]\\& = 
 \left[ {}_{A_{c(g)}}^{I_g} (I_g)_{A_{c(g)}}^{I_g}  \right]
\left[ {}_{A_{c(g)}}^{I_{gh}} ( P_{gh} )_{A_{d(h)}}^{I_{(gh)^{-1}}} \right]\\&
=\left[ {}_{A_{c(g)}}^{ \gamma_{I_g}(\epsilon_g \epsilon_{gh}) A_{c(g)} }
( I_g \otimes_{A_{c(g)}} 
P_{gh} )_{A_{d(h)}}^{ \gamma_{P_{gh}}^{-1}( \epsilon_g \epsilon_{gh} ) A_{d(h)} } \right]
\\&=\left [ {}_{A_{c(g)}}^{ \epsilon_g \epsilon_{gh} A_{c(g)} }
( I_g \otimes_{A_{c(g)}} 
P_{gh} )_{A_{d(h)}}^{ \gamma_{P_{gh}}^{-1}( \epsilon_g \epsilon_{gh} ) A_{d(h)} } \right]   \end{align*}
and
\begin{align*}F(gh)F(h^{-1})F(h) &= 
\left[ {}_{A_{c(g)}}^{I_{gh}} ( P_{gh} )_{A_{d(h)}}^{I_{(gh)^{-1}}} \right] 
\left[ {}_{A_{d(h)}}^{ I_{h^{-1}} } ( P_{h^{-1}} )_{A_{c(h)}}^{I_h} \right]
\left[ {}_{A_{c(h)}}^{I_h} (P_h)_{A_{d(h)}}^{I_{h^{-1}}} \right]\\& = 
 \left[ {}_{A_{c(g)}}^{I_{gh}} ( P_{gh} )_{A_{d(h)}}^{I_{(gh)^{-1}}} \right]
\left[ {}_{A_{d(h)}}^{I_{h^{-1}}} (I_{h^{-1}})_{A_{d(h)}}^{I_{h^{-1}}}  \right]\\&=\left[ {}_{ A_{c(g)} }^{ \gamma_{P_{gh}}( \epsilon_{(gh)^{-1}} \epsilon_{h^{-1}} ) A_{c(g)} }
( P_{gh} \otimes_{A_{d(h)}} 
I_{h^{-1}} )_{A_{d(h)}}^{ \epsilon_{(gh)^{-1}} \epsilon_{h^{-1}} A_{d(h)} }\right].\end{align*}
Thus, from the equality $F(g) F(g^{-1}) F(gh) =
F(gh) F(h^{-1}) F(h)$ and Proposition~\ref{isomorphism}, we get that
$\gamma_{P_{gh}} (\epsilon_{ (gh)^{-1} }\epsilon_{h^{-1}}) =
\epsilon_g \epsilon_{gh}$. Finally,
$\epsilon_g P_{gh} = 
\epsilon_g \epsilon_{gh} P_{gh} = 
\gamma_{P_{gh}} (\epsilon_{ (gh)^{-1} }\epsilon_{h^{-1}}) P_{gh} =
P_{gh} \epsilon_{ (gh)^{-1} }\epsilon_{h^{-1}} =
P_{gh} \epsilon_{h^{-1}}.$
\end{proof}

\begin{rem}
Let $F : G \rightarrow {\rm PIC}_{cat}$ be a partial functor of inverse categories. 
Then for every $(g,h) \in G_2$ there are $A_{c(g)}$-$A_{d(h)}$-bimodule isomorphisms 
\begin{displaymath}
	P_g \otimes_{A_{d(g)}} P_h \cong P_g  \otimes_{A_{d(g)}} P_{g\m} 
\otimes_{A_{c(g)}} P_{gh}\cong I_g\otimes_{A_{c(g)}} P_{gh} \cong \epsilon_g P_{gh}.
\end{displaymath}
\end{rem}

\begin{defi}\label{defpartialgeneralizedepsiloncrossedproduct}
Let $F : G \rightarrow {\rm PIC}_{cat}$ 
be a partial functor of inverse categories.
A \emph{partial factor set associated with $F$} is a family
$f = \{ f_{g,h} \mid (g,h) \in G_2 \}$, where each 
$f_{g,h} : P_g \otimes_{A_{d(g)}} P_h \rightarrow 
\epsilon_g P_{gh} = P_{gh} \epsilon_{h^{-1}}$ is an
isomorphism of $A_{c(g)}$-$A_{d(h)}$-bimodules, making the
following diagram commutative
\begin{equation}\label{ASSOCIATIVE}
\CD P_g \otimes_{A_{d(g)}} P_h \otimes_{A_{d(h)}} P_r @> {\rm id}_{P_g} 
\otimes f_{h,r} >> P_g \otimes_{A_{d(g)}} P_{hr} \epsilon_{r^{-1}} \\
@V f_{g,h} \otimes {\rm id}_{P_r} VV @VV f_{g,hr} V \\
\epsilon_g P_{gh} \otimes_{A_{d(h)}} P_r
@> f_{gh,r} >> \epsilon_g P_{ghr} \epsilon_{r^{-1}} \\
\endCD
\end{equation}
for all $(g,h,r) \in G_3$. If $f$ is a partial factor set
associated with $F$, then we define the partial generalized epsilon-crossed product
$(F,f)$ as the additive group $\oplus_{g \in G_1} P_g$
with multiplication defined by the biadditive extension of the relations
$x \cdot y = f_{g,h}(x \otimes y)$, if $(g,h) \in G_2$, 
and $x \cdot y = 0$, otherwise, for all $x \in P_g$
and $y \in P_h$ and all $g,h \in G_1$.
It is clear that if for each $g \in G_1$
we put $(F,f)_g = P_g$, then $(F,f)$ is a groupoid graded ring.
\end{defi}

\begin{rem} We notice that the notion of partial factor sets already exists in the literature, in a close but different sense. Indeed, partial projective representations and their corresponding factor 
sets where introduced in \cite{DN}. Later, in \cite{DN2}, 
these factor sets were called \emph{partial factor sets}. 
For a detailed account of these notions, we refer the reader to the survey \cite{pinedo2015}.
\end{rem}

\begin{prop}\label{firstcorrespondence}  
If $F : G \rightarrow {\rm PIC}_{cat}$ is a partial functor of 
inverse categories, then the ring $(F,f)$ is epsilon-strongly $G$-graded.
\end{prop}

\begin{proof}
Put $S = (F,f)$. By \eqref{ASSOCIATIVE} the
multiplication is associative. 
By the definition of the multiplication, for all $(g,h) \in G_2$,
the equality $S_g S_h = \epsilon_g S_{gh} = S_{gh} \epsilon_{h^{-1}}$ holds. 
All that is left to show is that $S$ has a multiplicative identity element. 
Take $e \in G_0$ and put $c_e = f_{e,e}( 1_{A_e} \otimes 1_{A_e} )$.
Take $a_e \in A_e$. Then, since $f_{e,e}$ is an $A_e$-bimodule homomorphism, 
it follows that 
\begin{align*}
a_e c_e &= 
a_e f_{e,e}( 1_{A_e} \otimes 1_{A_e} ) =
f_{e,e}( a_e \otimes 1_{A_e} ) =
f_{e,e}(1_{A_e} \otimes a_e) \\
&= f_{e,e}( 1_{A_e} \otimes 1_{A_e} ) a_e = c_e a_e.
\end{align*}
Thus, $c_e \in Z(A_e)$.
Since $f_{e,e}$ is surjective there are $a,a' \in A_e$
such that $f_{e,e}( a \otimes a' ) = 1_{A_e}$.
By $A_e$-bilinearity it follows that $aa' c_e = c_e aa' = 1_{A_e}$.
Hence $c_e \in U( Z(A_e) )$.
Now set $n_e = c_e^{-1}$.
Then $f_{e,e}(n_e \otimes n_e) = n_e$. 
Hence $n := \sum_{e \in G_0} n_e$ is a
multiplicative identity element of $S$. 
In fact, take $g \in G_1$ and $x \in P_g$.
Then there is $y \in P_g$ such
that $x = f_{c(g),g}(n_{c(g)} \otimes y)$. Thus, by
\eqref{ASSOCIATIVE}, we get that
\begin{displaymath}
\begin{array}{rcl}
 n \cdot x & = & n_{c(g)} \cdot x = f_{c(g),g}(n_{c(g)} \otimes x)
    = f_{c(g),g}(n_{c(g)} \otimes f_{c(g),g}(n_{c(g)} \otimes y)) \\
   & = &f_{c(g),g}(f_{c(g),c(g)} (n_{c(g)} \otimes n_{c(g)}) \otimes y)
    = f_{c(g),g}(n_{c(g)} \otimes y)
    = x.
\end{array}	
\end{displaymath}
Analogously, $x \cdot n = x$.
\end{proof}

\begin{prop}\label{secondcorrespondence}
If $S$ is an epsilon-strongly graded ring, then there is a 
partial functor of inverse categories
$F : G \rightarrow {\rm PIC}_{\rm cat}$ and a partial factor set $f$
associated with $F$ such that $S = (F,f)$.
\end{prop}

\begin{proof}
Define $F : G \rightarrow {\rm PIC}_{\rm cat}$ by 
$F(g) = [ {}_{S_{c(g)}}^{\epsilon_g S_{c(g)} }
S_g {}_{S_{d(g)}}^{\epsilon_{g\m} S_{d(g)} }]$, $g \in G_1$, 
and a partial factor set $f$
associated with $F$ by the multiplication maps $f_{g,h} :
S_g \otimes_{S_{d(g)}} S_h \rightarrow
\epsilon_g S_{gh} = S_{gh} \epsilon_{h^{-1}}$
for $(g,h) \in G_2$. The claim now
follows immediately from Proposition~\ref{tensor}(c).
\end{proof}

\begin{defi}
Let $F$ and $F'$ be partial functors of inverse categories 
from $G$ to ${\rm PIC}_{\rm cat}$ that coincide on $G_0$. 
Take partial factor sets $f$ and $f'$
associated with $F$ and $F'$ respectively and put $F(g) =
\left[ {}_{A_{c(g)}}^{I_g} (P_g)_{A_{d(g)}}^{J_g} \right]$, 
$F'(g) = \left[ {}_{A_{c(g)}}^{I'_g} (P_g')_{A_{d(g)}}^{J'_g} \right]$, $g\in G_1$.
A \emph{morphism from $f$ to $f'$} is a family $\alpha =
(\alpha_{g})_{g \in G_1}$, where each
$\alpha_{g} : P_{g} \rightarrow P_{g}'$ is an
$A_{c(g)}$-$A_{d(g)}$-bimodule homomorphism, such that
the diagram 
\begin{equation}\label{MORPHISM}
\CD P_g \otimes_{A_{d(g)}} P_{h} @> f_{g,h} >> \epsilon_gP_{gh} \\
@V \alpha_{g} \otimes \alpha_{h} VV @VV \alpha_{gh} V \\
P_{g}' \otimes_{A_{d(g)}} P_{h}'
@> f_{gh}' >> \epsilon_g P_{gh}' \\
\endCD
\end{equation}
is commutative
for all $(g,h) \in G_2$.
\end{defi}

\begin{lem}\label{alpha} 
With the above notation, a morphism $\alpha$ from $F$ to $F'$ 
induces a homomorphism of graded rings $\alpha$ from $(F, f)$ to $(F', f').$ 
Moreover, if each $\alpha_e$, $e \in G_0$, is surjective, 
then $\alpha(1) = 1$. The map $\alpha$ is an isomorphism if
and only if each $\alpha_e$, $e \in G_0$, is bijective.
\end{lem}

\begin{proof}
Similar to the proof of \cite[Lemma 4.2]{lundstrom2004}.
\end{proof}

\begin{prop}
The isomorphism class of $(F, f)$ does not depend
on the choice of the bimodules $P_g$. 
\end{prop}

\begin{proof}
Put $F(g) = \left[ {}_{A_{c(g)}}^{I_g} (P_g)_{A_{d(g)}}^{J_g} \right] = 
\left[ {}_{A_{c(g)}}^{I_g} (P_g')_{A_{d(g)}}^{J_g} \right]$, for $g \in G_1$. 
Then there exists an $A_{c(g)}$-$A_{d(g)}$-bimodule homomorphism 
$\alpha_{g} : P_{g} \rightarrow P_{g}'$. 
If we now put $f'_{g,h} = \alpha^{-1}_{gh}\circ f_{gh}\circ 
(\alpha_g\otimes \alpha_h),$ for $(g,h)\in G_2$, then $f'$ is
a factor set associated with $F$ and \eqref{MORPHISM} commutes.
\end{proof}

\section{Partial Cohomology of groupoids}\label{sectionpartialcohomology}

In this section, we extend the construction of a partial 
cohomology theory for partial actions of groups on 
commutative monoids, from \cite{DKh}, to partial actions of groupoids. 
We follow closely the presentation and the proofs in \cite{DKh}.
Partial actions of groupoids on rings were
first studied in \cite{bagio2012}. 
Partial actions of categories on sets and topological spaces
have been introduced in \cite{nystedt2017}.
For the rest of this section,
let $G$ be a groupoid and suppose that
$B$ is the product of a collection of commutative semigroups $\{ B_e \}_{e \in G_0}$.

\begin{rem}
Let $\{ \theta_g \}_{g \in G_1}$
be a unital  partial groupoid action of $G$ on $B$\footnote{In the sense of Definition~\ref{defiunitalpartialaction}, but in this case $\{\theta_g\}_{g\in G_1}$ is a family of semigroup isomorphisms.}
Note that if $e \in G_0$ and $C$ and $D$ are unital 
ideals of $B_e$, then $C \cap D = CD$ so it follows from 
\cite[Lemma 1.1]{bagio2012} that the properties (G2) and (G3) can be replaced by
\begin{itemize}

\item[(G2$'$)] if $(g,h) \in G_2$, then $\theta_g ( B_{g^{-1}} B_h ) = B_g B_{gh}$, and

\item[(G3$'$)] if $(g,h) \in G_2$ and $x \in B_{h^{-1}} B_{h^{-1} g^{-1}}$,
then $\theta_g( \theta_h ( x ) ) = \theta_{gh} (x)$,

\end{itemize}
respectively. 

A unital partial $G$-module is a pair $(B,\theta),$ 
where $B$ is a commutative monoid and $\theta$ is a unital partial action of $G$ on $B$.  
\end{rem}

\begin{defi}
A morphism $\psi\colon (B,\theta)\to (B',\theta')$ of 
partial $G$-modules is a set of monoid homomorphisms
$\psi=\{\psi_{c(g)}:B_{c(g)}\to B'_{c(g)}\}_{g\in G}$ such that
\begin{itemize}
\item $\psi_{c(g)}( B_g)\subseteq B'_g,$
\item $\theta'_g \circ \psi_{c(g)}=\psi_{c(g)} \circ \theta_g \,\,\text{ on}\,\, B_{g\m},$
\end{itemize}
for all $g\in G_1.$

Recall that, if $n \geq 2$, then we let $G_n$ denote the set of
all $(g_1,\ldots,g_n) \in \times_{i=1}^n G_1$
that are composable, that is, such that
for every $i \in \{ 1,\ldots,n-1 \}$ the relation $d(g_i) = c(g_{i+1})$ holds.

We denote by $pMod(G)$ the category of (unital) partial $G$-modules. 
Sometimes, for convenience  $(B,\theta)$ will be replaced by $B.$
Suppose that $B\in pMod(G).$ 
An \emph{$n$-cochain of $G$ with values in $B$}
is a function $f : G_n \rightarrow B$ such that
for every $(g_1,\ldots,g_n) \in G_n$,
$f(g_1,\ldots,g_n)$ is an invertible element of 
$B_{(g_1,\ldots,g_n)} = 
B_{g_1} B_{g_1 g_2} \cdots B_{g_1 \cdots g_n}$. 
Denote the set of $n$-cochains by $C^n(G,B)$.
We let $C^0(G,B)$ denote $U(B)$, the group of units in $B$.
\end{defi}

\begin{prop} Let $B\in pMod(G).$ Then
$C^n(G,B)$ is an abelian group under pointwise multiplication.
\end{prop}

\begin{proof}
This is clear if $n=0$. Now suppose that $n \geq 1$.
Define $e_n \in C^n(G,B)$ in the following way.
Given $(g_1,\ldots,g_n) \in G_n$, put
$e_n(g_1,\ldots,g_n) = 1_{g_1} 1_{g_1 g_2} \cdots 1_{g_1 \cdots g_n}$.
It is clear that $e_n$ is an identity element of $C^n(G,B)$.
Given $f \in C^n(G,B)$ and $(g_1,\ldots,g_n) \in G_n$, put
$f^{-1}(g_1,\ldots,g_n) = f(g_1,\ldots,g_n)^{-1}$,
where the inverse is taken in $B_{(g_1,\ldots,g_n)}$.
It is clear that $ff^{-1} = f^{-1}f = e_n$.
\end{proof}

\begin{defi}
Let  $B\in pMod(G)$
and let $n$ be a non-negative integer.
Define a map $\delta_n : C^n(G,B) \rightarrow C^{n+1}(G,B)$
in the following way.
Take $b = (b_e)_{e \in \ob(G)} \in U(B)$ and $g \in G_1$.
Put 
\begin{displaymath}
	\delta^0(b)(g) = \theta_g ( 1_{g^{-1}} b_{d(g)}) b_{c(g)}^{-1}.
\end{displaymath}
Now suppose that $n$ is a positive integer.
Take $f \in C^n(G,B)$ and $(g_1,\ldots,g_{n+1})$ in $G_{n+1}$.
Put
{\small
\begin{align*}
	&\delta^n(f)(g_1,\ldots,g_{n+1}) = \\ 
& \theta_{g_1}( 1_{g_1^{-1}} f(g_2,\ldots,g_{n+1}) )
\left( \prod_{i=1}^n f(g_1,\ldots,g_i g_{i+1}, \ldots g_{n+1})^{ (-1)^i } \right) 
f(g_1,\ldots,g_n)^{ (-1)^{n+1} }.
\end{align*}
}%
\end{defi}

Adapting the proofs of \cite[Proposition 1.5]{DN} and  
\cite[Proposition 1.7]{DN} to our situation, we get the following.

\begin{prop}
Suppose that $B\in pMod(G),$
and that $n$ is a non-negative integer. Then the following assertions hold:
\begin{itemize}
\item[(a)] The map $\delta^n$ is a well-defined homomorphism of groups
$C^n(G,B) \rightarrow C^{n+1}(G,B)$ satisfying
$\delta^{n+1} \delta^n = e_{n+2}$;
\item[(b)] The map sending $B$ to the sequence $\{\delta^n\colon C^n(G,B) 
\rightarrow C^{n+1}(G,B)\}_{n\in \N}$ is a functor from 
$pMod(G)$ to the category of complexes of abelian groups.
\end{itemize}
\end{prop}

\begin{defi}
Let  $B\in pMod(G)$
and let $n$ be a positive integer.
The map $\delta^n$ is called a \emph{coboundary homomorphism}.
We define the abelian groups $Z^n(G,B) = {\rm ker}( \delta^n )$,
$B^n(G,B) = {\rm im}( \delta^{n-1} )$.
The quotient group $H^n(G,B) = Z^n(G,B) / B^n(G,B)$
is called \emph{the $n$th cohomology group of $G$ with values in $B$}.
We put $H^0(G,B) = Z^0(G,B) = {\rm ker}( \delta^0 )$.
\end{defi}

Let $G$ be a groupoid. Denote by $I(X)_{cat}$  the  subcategory of 
BIJ$_{\rm cat}$ having as objects the collection  
$(X_e)_{e \in G_0}$ of abelian semigroups and morphisms 
$[_{X_{c(g)}}^{X_g}{f_g}_{X_{d(g)}}^{X_{g\m}}],$ where $X_g$ is 
a unital ideal of $X_{c(g)}$ and  $f_g\colon X_{g\m}\to X_g$,  
is a monoid isomorphism for all $g \in G_1.$ 
The composition in $I(X)_{{\rm cat}}$ is defined in the same way as in  
ISO$_{cat},$ the map  $* : (I(X)_{cat})_1 \rightarrow ( I(X)_{cat})_1$ 
is also defined by restriction of the map $*$ defined on $( {\rm BIJ}_{cat} )_1$. 
It follows from Proposition~\ref{BIJCAT} that $I(X)_{cat}$ is an inverse category.

\begin{prop} 
If $G$ is a groupoid and $X = \prod_{e \in G_0} X_e$, 
then $X \in pMod(G),$ if and only if, there is a partial functor of 
inverse categories $F \colon G \to I(X)_{cat}.$ 
\end{prop}

\begin{proof} 
First we show the ''only if'' statement.
Let $\{\theta_g\colon X_{g\m}\to X_g\}_{g \in G_1}$ 
be a partial action of $G$ on $X$ and define  
$F\colon G\to I(X)_{cat},$ by  
$F(g)=[_{X_{c(g)}}^{X_g}{\theta_g}_{X_{d(g)}}^{X_{g\m}}],$ 
for $g \in G_1,$ and $F_e=X_e,$ for any $e \in G_0$. 
Then $F$ is a partial functor of inverse categories.

Now we show the ''if'' statement.
Let $F\colon G\to I(X)_{{\rm cat}}$ be a partial functor of inverse categories.
Put $F(g) = [_{X_{c(g)}}^{X_g}{\theta_g}_{X_{d(g)}}^{X_{g\m}}],$ for $g \in G_1$. 
We shall show that the family $\{\theta_g\}_{g \in G_1}$ gives a partial action 
of $G$ on  $X=\prod_{e \in G_0} X_e.$ 
It is clear that for each $g \in G_1$, $B_g$ is an ideal of 
$B_{c(g)},$ $B_{c(g)}$ is an ideal of $B$ and
$\theta_g : B_{g^{-1}} \rightarrow B_g$ is a monoid isomorphism.   
Now we check (G1), (G2) and (G3) from Definition~\ref{defipartialaction}.

(G1): Let $e \in G_0.$ Then 
$F(e) =\left [ {}_{X_{e}}^{X_e} (\theta_e)_{X_{e}}^{X_e} \right],$  
and $\theta_e$ is the identity map on $X_e.$ 

(G2)-(G3): Let $(g,h)\in G_2.$ Then $F(g) F(h)  = F(gh) F(h)^* F(h),$ 
which by the definition of $F$ implies
$\theta_g\circ\theta_h=
\theta_{gh}\circ\theta\m_h\circ\theta_h=
\theta_{gh}\circ{\rm id}_{X_{h\m}},$ 
which in turn implies that $ \theta_{gh}$ is an extension of  $\theta_g \circ\theta_h.$
\end{proof}

Let $F : G \rightarrow {\rm PIC}_{cat}$ be a partial functor
of inverse categories.
For each $e \in G_0$ define the ring $A_e$ by $F(e) = A_e$. 
Take $g \in G_1$. Put
$F(g) =\left [ {}_{A_{c(g)}}^{I_g} (P_g)_{A_{d(g)}}^{I_{g^{-1}}} \right]$.
Define $B_g = Z(I_g)$ and $B = \prod_{e \in \ob(G)} Z(A_e)$.

\begin{prop}
With the above notation we have 
$B \in pMod(G).$
\end{prop}

\begin{proof} 
By Proposition~\ref{lfunctor} there is a partial functor
of inverse categories $L : {\rm PIC}_{cat} \rightarrow {\rm ISOC}_{cat}$.
From Proposition~\ref{composition} it follows that 
$l = L \circ F \colon G \to {\rm ISOC}_{cat}$ 
is a partial functor of inverse categories. 
But $l(g)=\left [ {}_{B_{c(g)}}^{B_g} (\gamma_{P_g})_{B_{d(g)}}^{B_{g^{-1}}} \right],$
so $l\colon G\to I(X)_{cat}$ and hence we get that $B\in pMod(G).$
\end{proof}

\section{Proof of the main result}\label{proofmain}

In this section, we prove Theorem~\ref{maintheoremepsilon} 
which was stated in Section~\ref{sec:intro}.
For the convenience of the reader, we shall now recall its exact wording.

\begin{thmnonum}
If $G$ is a groupoid,
$F : G \rightarrow {\rm PIC}_{cat}$ is a partial functor of inverse categories
and $f$ is a partial factor set associated with $F$,
then the map $H^2 ( G , U(Z(A)) ) \rightarrow C(A,F)$,
defined by $[q] \mapsto qf$, is bijective.
\end{thmnonum}

In order to prove the above theorem, we need the following result.

\begin{prop}\label{proplast}
Let $f$ and $f'$ be factor sets associated with $F$.
\begin{itemize}
\item[{\rm (a)}] If $q \in Z^2(G , B)$, then $fq$ is a factor
set associated with $F$.

\item[{\rm (b)}] There is $q \in Z^2(G , B)$ such that $f' =
qf$.

\item[{\rm (c)}] A cocycle $q \in Z^2(G , B)$ belongs to
$B^2(G , B)$ if and only if there is a graded ring
isomorphism $\alpha$ from $(F,f)$ to $(F,qf)$ such that each
for all $g \in G_1$ the graded restriction $\alpha_g$ to $P_g$
is an $A_{c(g)}$-$A_{d(g)}$-bimodule isomorphism.

\item[(d)] The map from $Z^2(G , B)$ to the collection of factor
sets associated with $F$, defined by $q \mapsto qf$, is bijective.  
\end{itemize}
\end{prop}

\begin{proof}
(a) Put $f'' = fq$. We need to verify that
\eqref{ASSOCIATIVE} commutes for $f''$. Take $(g,h,p) \in G_3$, 
$x \in P_g$, $y \in P_h$ and $z \in P_p$. Then
{\small
\[
\begin{array}{rcl}
(f_{gh,p}'' \circ (f_{g,h}'' \otimes {\rm id}_{P_p} ) ) (x \otimes y \otimes z) 
&=& q_{gh,p} q_{g,h} (f_{gh,p} \circ (f_{g,h} \otimes {\rm id}_{P_p} )) (x \otimes y \otimes z) \\
&=& (q_{g,hp} \gamma_{P_g}(q_{h,p}))
( f_{g,hp} \circ ( {\rm id}_{P_g} \otimes
f_{h,p}) ) (x \otimes y \otimes z) \\
&=& (q_{g,hp} \gamma_{P_g}(q_{h,p}))
f_{g,hp} (x \otimes f_{h,p}(y \otimes z)) \\
&=& f_{g,hp}'' ( \gamma_{P_g}( q_{h,p} ) x \otimes f_{h,p}(y \otimes z))) \\
&=& f_{g,hp}'' (x q_{h,p} \otimes f_{h,p}(y \otimes z)) \\
&=& f_{g,hp}''(x \otimes f_{h,p}'' (y \otimes z)) \\
&=& (f_{g,hp}'' \circ ({\rm id}_{P_{\sigma}} \otimes f_{h,p}'')) (x \otimes y \otimes x).
\end{array}
\] 
}

(b) Take $(g,h) \in G_2$.
Then $f_{g,h}' \circ f_{g,h}^{-1}$ is an $A_{c(\sigma)}$-$A_{d(\sigma)}$-bimodule
automorphism of $P_{gh}$. Hence, by Proposition~\ref{propequivalence}, 
there is $q_{g,h} \in U( Z( I_{c(g)} ) )$ such that
$(f_{g,h}' \circ f_{g,h}^{-1}) (x) = q_{\sigma,\tau} x$, for $x \in P_{gh}$.
By \eqref{ASSOCIATIVE} it follows that $q \in Z^2(G,B)$.

(c) Suppose now that $q \in B^2(G,A)$. Then there is $c \in C^1(G,A)$ 
such that for all $(g,h) \in G_2$ it follows that
$q_{g,h} = \gamma_{P_g}(c_h) c_g c_{gh}^{-1}$.
Define a map $\alpha$ from $(F,qf)$ to $(F,f),$ by 
$\alpha(x) = c_g x $, for $x \in P_g$. 
If $x \in P_g$, $y \in P_h$ and $(g,h) \in G_2$, then 
$\alpha(xy) = 
c_{gh} q_{g,h} f_{g,h}(x \otimes y) =
q_{g,h}^{-1} \gamma_{P_g}(c_h) c_g q_{g,h} f_{g,h} (x \otimes y) = 
f_{g,h} ( \gamma_{P_g} x \otimes c_h y) = 
\alpha(x)\alpha(y)$. 
Clearly, for all $g \in G_1$, the map $\alpha_g$ is an
$A_{c(g)}$-$A_{d(g)}$-bimodule isomorphism.

On the other hand, suppose that there is an isomorphism of graded
rings $\beta$ from $(F,qf)$ to $(F,f)$ such that for all $g \in G_1$
the map $\beta_g$ is an
$A_{c(g)}$-$A_{d(g)}$-bimodule isomorphism. 
From Proposition~\ref{propequivalence} it follows that
for each $g \in G_1$ there is $d_g \in U(Z(I_{ c(g) }) )$ such that
for all $x \in P_g$ the equation $\beta_g(x) = d_g x$ holds.
Therefore, for all $x \in P_g$, $y \in P_h$ and all 
$(g,h) \in G_2$, we get that 
$\beta(xy) = \beta(x)\beta(y) \Leftrightarrow
d_{gh} q_{g,h} f_{g,g}(x \otimes y) =
f_{g,h}(d_g x \otimes d_h y) \Leftrightarrow
d_{gh} q_{g,h} f_{g,h} (x \otimes y) =
d_g \gamma_{P_g} (d_h) f_{g,h} (x \otimes y)$. 
Thus, $q \in B^2(G,A)$.

(d) This follows from (a), (b) and (c).
\end{proof}

\begin{defi}
If $f$ and $f'$ are factor sets associated with $F$,
then we write $(F,f) \approx (F,f')$ if there is an isomorphism of
graded rings from $(F,f)$ to $(F,f')$ such that each graded
restriction to $P_g$, $g \in G_1$, is an
$A_{c(g)}$-$A_{d(g)}$-bimodule isomorphism. 
Let $C(A,F)$ denote the collection of equivalence classes of 
generalized groupoid epsilon-crossed products $(F,f)$ 
modulo $\approx$, where $f$ runs over all factor sets associated with $F$.
\end{defi}

\noindent {\bf Proof of Theorem~\ref{maintheoremepsilon}.}
This follows immediately from Proposition~\ref{proplast}. 
\qed

\end{document}